\documentclass[a4paper,11pt,reqno]{amsart}

\usepackage{amsfonts}
\usepackage{amssymb}
\usepackage{amscd}
\usepackage{amsthm}
\usepackage{a4wide}
\usepackage{mathrsfs}
\usepackage[applemac]{inputenc}
\usepackage{amsmath}
\usepackage{amssymb}
\usepackage{amsthm}
\usepackage{textcomp}
\usepackage{graphicx}
\usepackage{enumerate}
\usepackage{mathrsfs}
\usepackage{frcursive}
\usepackage{tikz}
\usepackage[cyr]{aeguill}
\usepackage{xspace}
\usepackage{hyperref}
\usepackage{appendix}
\usepackage{esint}
\usepackage{cancel}
\usepackage{calc}
\usepackage[colorinlistoftodos,prependcaption]{todonotes}
\usepackage{mathtools}
\usepackage{esint}

\newtheorem{defn}{Definition}[section]
\newtheorem{lemma}[defn]{Lemma}
\newtheorem{prop}[defn]{Proposition}
\newtheorem{theo}[defn]{Theorem}
\newtheorem{coro}[defn]{Corollary}

\newtheorem{rk}[defn]{Remark}

\newcommand{\Id}{Id}

\newcommand{\N}{\ensuremath{{\mathbb N}}}

\def\Sc{\rm R}

\def\al{\alpha}

\def\downto{ \searrow }

\def\Ric{\mathop{\rm Ric}\nolimits}

\def\Rm{\mathop{\rm Rm}\nolimits}
\def\tr{\mathop{\rm tr}\nolimits}

\def\vol{\mathop{\rm vol}\nolimits}
\def\eucl{\mathop{\rm eucl}\nolimits}

\def\vol{\mathop{\rm Vol}\nolimits}

\def\inj{\mathop{\rm inj}\nolimits}

\def\div{\mathop{\rm div}\nolimits}

\def\AVR{\mathop{\rm AVR}\nolimits}
\def\Id{\mathop{\rm Id}\nolimits}

\def\limGH{\mathop{\rm limGH}\nolimits}
\def\Li{\mathop{\rm \mathscr{L}}\nolimits}

\def\limGH{\mathop{\rm limGH}\nolimits}

\def\Ric{\mathop{\rm Ric}\nolimits}
\def\Rc{\mathop{\rm Ric}\nolimits}

\def\Rm{\mathop{\rm Rm}\nolimits}

\def\tr{\mathop{\rm tr}\nolimits}

\def\vol{\mathop{\rm vol}\nolimits}
\def\eucl{\mathop{\rm eucl}\nolimits}

\def\vol{\mathop{\rm Vol}\nolimits}

\def\inj{\mathop{\rm inj}\nolimits}

\def\div{\mathop{\rm div}\nolimits}

\def\AVR{\mathop{\rm AVR}\nolimits}
\def\Id{\mathop{\rm Id}\nolimits}

\def\limGH{\mathop{\rm limGH}\nolimits}
\def\Li{\mathop{\rm \mathscr{L}}\nolimits}

\def\limGH{\mathop{\rm limGH}\nolimits}

\def\ep{\varepsilon}
\def\grad{\nabla}

\def\R{{\rm R}}

\def\Xint#1{\mathchoice
{\XXint\displaystyle\textstyle{#1}}%
{\XXint\textstyle\scriptstyle{#1}}%
{\XXint\scriptstyle\scriptscriptstyle{#1}}%
{\XXint\scriptscriptstyle\scriptscriptstyle{#1}}%
\!\int}
\def\XXint#1#2#3{{\setbox0=\hbox{$#1{#2#3}{\int}$ }
\vcenter{\hbox{$#2#3$ }}\kern-.6\wd0}}

\def\dashint{\Xint-}
\def\Ob{\mathcal{T}}

\def\lap{\Delta}
\def\partt{\frac{\partial}{\partial t}}
\def\de{\delta}

\def\al{\alpha}

\def\Xint#1{\mathchoice
{\XXint\displaystyle\textstyle{#1}}%
{\XXint\textstyle\scriptstyle{#1}}%
{\XXint\scriptstyle\scriptscriptstyle{#1}}%
{\XXint\scriptscriptstyle\scriptscriptstyle{#1}}%
\!\int}
\def\XXint#1#2#3{{\setbox0=\hbox{$#1{#2#3}{\int}$ }
\vcenter{\hbox{$#2#3$ }}\kern-.6\wd0}}

\def\dashint{\Xint-}

\author{Alix Deruelle}
\address{Universit\'e Paris-Saclay, CNRS, Laboratoire de math\'ematiques d'Orsay, 91405, Orsay, France}
\email{alix.deruelle@universite-paris-saclay.fr}
\author{Felix Schulze}
\address{Mathematics Institute, Warwick University, Gibbet hill road, Coventry CV4 7AL, UK}
\email{felix.schulze@warwick.ac.uk}
\author{Miles Simon}
\address{Institut f\"ur Analysis und Numerik (IAN), Universit\"at Magdeburg, Universit\"atsplatz~2\\${\ }\ \, $ 39106 Magdeburg, Germany }
\email{msimon@ovgu.de}

\begin{document}

\title{On the Hamilton-Lott conjecture in higher dimensions}

% General info
\subjclass[2000]{}

\dedicatory{}

\keywords{}

\begin{abstract}  
We study $n$-dimensional Ricci flows with non-negative Ricci curvature where the curvature is pointwise controlled by the scalar curvature and bounded by $C/t$, starting at metric cones which are Reifenberg outside the tip. We show that any such flow behaves like a self-similar solution up to an exponential error in time.  As an application, we show that smooth $n$-dimensional complete non-compact  Riemannian manifolds which are uniformly PIC1-pinched, with positive asymptotic volume ratio, are Euclidean. This confirms a higher dimensional version of a conjecture of Hamilton and Lott under the assumption of non-collapsing. It also yields a new and more direct proof of the original conjecture  of  Hamilton and Lott in three dimensions.  
\end{abstract}

\maketitle

\tableofcontents

\date{\today}

\section{Introduction}
\subsection{Overview}\label{overview}
In this paper, we consider smooth, complete solutions $(M^n,g(t))_{t\in (0,T)}$ to the Ricci flow defined on smooth, connected manifolds satisfying for $t\in(0,T)$,
\begin{equation}\label{curv-cond-1}
\Rc(g(t)) \geq 0\quad \text{ and }\quad  |\Rm(g(t))| \leq \frac{D_0}{ t},
\end{equation} 
where $D_0$ is a positive 
constant. The curvature conditions \eqref{curv-cond-1} are particularly relevant since they are invariant under parabolic rescaling. Due to \cite{SiTo2} it is known that  \eqref{curv-cond-1} ensures the existence of an initial metric $d_0$ on $M$ (interpreted as a metric space) such that the flow converges back to it in the distance sense (see Section \ref{sec-2-rf}). 

This setting has been shown to occur in many situations,  a  prominent one being that of self-similar solutions (also known as expanding solitons) with non-negative curvature operator coming out of cones with non-negative curvature operator: see for example \cite{Sch-Sim}, \cite{Der-Smo-Pos-Cur-Con}, \cite{SiTo2}, \cite{Bam-Cab-Riv-Wil}. See also \cite{Bamler-Chen} for $4$-dimensional expanding solitons coming out of metric cones with non-negative scalar curvature.

The first result of this paper concerns solutions to Ricci flow satisfying \eqref{curv-cond-1} under the assumption that the scalar curvature controls the whole curvature tensor pointwise, starting from a sufficiently regular metric cone. It quantifies locally (in space) how far such a solution is from being self-similarly expanding. 
\begin{theo}\label{main-I}
Let $(M^n,g(t))_{t\in(0,T)}$ be a smooth, complete, connected Ricci flow 
such that there exists $D_0 \in 
(0,\infty)$ such that on $M\times(0,T)$:
\begin{equation*}
\Ric(g(t))\geq 0,\quad |\Rm(g(t))|_{g(t)}\leq \frac{D_0}{t}.
\end{equation*}
Assume that the pointed limit in the distance sense of $(M^n,d_{g(t)},o)$ as $t$ goes to $0$ is a metric cone $(C(X),d_0,o)$ that is uniformly locally $n$-Reifenberg outside its tip $o$.
Assume further that on $M\times(0,T)$ for some $D_1 \in 
(0,\infty)$,
$$|\Rm(g(t))|_{g(t)}\leq D_1\R_{g(t)},
$$ or, 
$$ C(X) \text{ is smooth away from } o.$$\\[-1ex]
Then there exists a smooth function $u:B_{d_0}(o,4)\times(0,T)\rightarrow \mathbb{R}$ satisfying the following properties:
\begin{enumerate}
\item (Equation) The function $u$ solves 
\begin{equation*}
\frac{\partial}{\partial t} u=\Delta_{g(t)}u-\frac{n}{2}
\end{equation*}
 on $B_{d_0}(o,4)\times(0,T)$. \\[-1ex]
\item (Initial condition) There exists $C \in 
(0,\infty)$ such that on $B_{d_0}(o,3)\times(0,T)$, 
\begin{equation*}
\left|u(\cdot,t)-\frac{d_0(o,\cdot )^2}{4}\right|\leq C\sqrt{t}.
\end{equation*}
\item (Local obstruction to be an expanding gradient Ricci soliton) For each non-negative integer $k$ there exists  $C_k \in 
(0,\infty)$ such that on $B_{d_0}(o,2)\setminus \overline{B_{d_0}(o,1)}\times(0,T)$:
\begin{equation*}
\left|\nabla^{g(t),k}\left(\nabla^{g(t),\,2}u-t\Ric(g(t))-\frac{g(t)}{2}\right)\right|_{g(t)}\leq e^{-{ \frac{C_k}{t}}}.
\end{equation*}
\item \label{grad-bd-main-I}(Gradient bound)
There exists  $C\in (0,\infty)$ such that on $B_{d_0}(o,2)\setminus \overline{B_{d_0}(o,1)}\times(0,T)$,
\begin{equation*}
\left||\nabla^{g(t)}u(t)|^2_{g(t)}-u(t)\right|\leq Ct.
\end{equation*}
\end{enumerate}

\end{theo}
A metric space is $n$-Reifenberg at a point if all tangent cones
 at the point exist and are Euclidean $n$-space. It is called uniformly locally Reifenberg, if every point is $n$-Reifenberg and the convergence to the Euclidean tangent cone is locally uniform.

Note that if $(M,d_0)$ is a metric cone which  is  obtained as the limit of smooth manifolds with bounded, non-negative curvature operator then it is known that a Ricci flow solution satisfying the conditions of the theorem exists, and is in fact an expanding gradient soliton:  see \cite{Sch-Sim}, \cite{Der-Smo-Pos-Cur-Con}.

 In dimension $3$ non-negative Ricci-curvature implies 
that the norm of the full curvature tensor is (up to a multiplicative universal constant) bounded by the scalar curvature.
Hence the assumption $|\Rm(g(t))|_{g(t)}\leq D_1\Sc_{g(t)}$ is always satisfied in three dimensions.

 In higher dimensions,   the same is true for Riemannian manifolds  which have  non-negative Ricci curvature
 and  non-negative isotropic  curvature: see for instance Appendix \ref{sec-curv-constraints} and the references therein. In particular, if $(M^n,g(t))_{t\in(0,T)}$ is weakly PIC1 then $(M^n,g(t))_{t\in(0,T)}$ satisfies both assumptions, i.e.~$\Ric(g(t))\geq 0$ and $|\Rm(g(t))|_{g(t)}\leq D_1\R_{g(t)}$. Weakly PIC1 is in turn implied by either $2$-non-negative curvature operator or weakly PIC2: see \cite{Top-Sur} for a survey on these curvature conditions.

We recall that the Hamilton-Lott conjecture states that a complete, connected 3-dimensional Riemannian manifold which is uniformly Ricci pinched is either compact or flat, see \cite[Conjecture 3.39]{Chow-Lu-Ni} and \cite[Conjecture 1.1]{Lott-Ricci-pinched} (with the additional assumption of bounded curvature). The second main result of this paper is motivated by the recent resolution of this conjecture by the authors \cite{DSS-II} and Lee-Topping \cite{Lee-Top-3d}. See also \cite{Hui-Koe} for a proof using inverse mean curvature flow. Recall that a Riemannian manifold $(M^n,g)$ is Ricci-pinched if $\Ric(g)\geq 0$ and if there exists a positive constant $c$ such that $\Ric(g)\geq c\,\R_g \,g$ in the sense of symmetric $2$-tensors. In \cite[Questions $1.5$]{DSS-II}, we asked whether such a conjecture holds in higher dimensions when the metric is not only Ricci-pinched but also $2$-pinched i.e.~if there exists a constant $c>0$ such that the sum of the two lowest eigenvalues $\lambda_i(g)$, $i=1,2,$ of the curvature operator satisfies $\lambda_1(g)+\lambda_2(g)\geq c\,\R_g$ on $M$. 

It is however legitimate to ask the same question either for Ricci-pinched manifolds in all dimensions or under other natural pinching conditions. Our main tool in \cite{DSS-II} and in this paper being the Ricci flow, it is natural to ask for a curvature condition that is preserved along the flow: note that non-negative Ricci curvature, respectively being Ricci pinched,   is not preserved in dimensions higher than $3$ in general.  Both facts can be seen in \cite{Maximo}, 
 where a smooth solution to Ricci flow
on a closed manifold is constructed which has strictly positive Ricci curvature everywhere  at time zero, but has negative Ricci curvature in some directions at some points at later times, see
 \cite[Theorem 1]{Maximo}.

A crucial intermediate result that can be seen as a dynamical version of Hamilton-Lott's conjecture is:
 \begin{theo}\label{main-III}
Let $(M^n,g(t))_{t>0}$ be a smooth, complete, connected  Ricci flow 
such that there exist uniform positive constants $D_0$ and $D_1$ such that on $M\times\mathbb{R}_+$:
\begin{equation*}
\Ric(g(t))\geq 0,\quad |\Rm(g(t))|_{g(t)}\leq \frac{D_0}{t},\quad |\Rm(g(t))|_{g(t)}\leq D_1\R_{g(t)}.
\end{equation*}Assume $(M^n,g(t))_{t>0}$ is uniformly Ricci-pinched and uniformly non-collapsed at all scales, i.e.~there exists $c>0$ such that for $t>0$, $\Ric(g(t))\geq c\,\R_{g(t)}\,g(t)$ on $M$ and there exists $V_0>0$ such that for \textit{some} $t>0$ and all $r>0$, $\vol_{g(t)}\mathop{B}_{g(t)}(p,r)\geq V_0r^n$. Then $(M^n,g(t))_{t>0}$ is isometric to Euclidean space.
 \end{theo}
 
Theorem \ref{main-III} answers affirmatively \cite[Questions $1.6$]{DSS-II} and (as detailed above) the assumptions on the curvature are implied by the curvature conditions ''weakly PIC1 and uniformly Ricci-pinched along the Ricci flow'' or "uniformly PIC1 pinched along the Ricci flow". For a formal definition of the notion PIC1 pinched, see \cite{Top-Sur} for instance.

Thanks to Theorem \ref{main-III}, we are able to answer \cite[Questions $1.5$]{DSS-II} under an additional non-collapsing assumption. We note that $2$-positive curvature operator implies PIC1 and $2$-pinched implies PIC1 pinched. The following result also confirms the question in  \cite[Remark 1.4]{Lee-Top-PIC2} under an additional non-collapsing assumption.

\begin{theo}\label{main-II}
Let $(M^n,g)$ be a smooth, complete, connected  Riemannian manifold that is PIC1 pinched. Assume it is non-collapsed at all scales: $\AVR(g):=\lim_{r\rightarrow +\infty}r^{-n}\vol_gB_g(p,r)>0$. Then $(M^n,g)$ is isometric to Euclidean space.
\end{theo}

Theorem \ref{main-II} provides a new proof of the Hamilton-Lott conjecture in dimension $3$, as detailed below. The starting point of the proof in \cite{DSS-II} are the following existence (E) and non-collapsing (NC) results of \cite{Lott-Ricci-pinched} for starting metrics $(M^3,g_0)$ which are non-flat, complete, connected with non-negative Ricci curvature and bounded curvature: \\[-1ex]
\begin{enumerate}
 \item  [(E)]there exists a smooth solution $(M^3,g(t))_{t\in [0,\infty)}$ to Ricci flow for all time and the solution remains uniformly  Ricci  pinched,  $\Ric(g(t)) \geq \al\,  \R_{g(t)} \,g(t) >0 $  for  some $\al>0,$ and $|\Rm(g(t))|_{g(t)}\leq c/t$ for $t\in (0,\infty).$\\[-1ex]
\item [(NC)] the solution is non-collapsed at all scales uniformly in time. More precisely, it has constant positive asymptotic volume ratio: $\AVR(g(t))= V_0>0$
for all $t \in [0,\infty)$.   \\[-1ex]
\end{enumerate}
Important ingredients in the  proof of \cite{DSS-II} are  a local-in-time stability theorem for the Ricci flow  (see \cite[Theorem $1.2$]{DSS-II}), existence results  for  self-similar solutions coming out of non-negatively curved $3$-dimensional Alexandrov metric cones  and a number of non-trivial results from the theory of RCD spaces. The proof of Theorem \ref{main-II} in this paper, which relies both on Theorems \ref{main-I} and \ref{main-III},  does not require any of these ingredients. 

Assuming the initial metric to be Ricci-pinched for $n=3$ or PIC1 pinched for $n\geq 4$ (for $n=3$, PIC1 pinching is equivalent to Ricci pinching), the existence part (E) was extended by Lee-Topping, allowing the initial metric to have unbounded curvature. 

\begin{theo}[$\text{\cite[Theorem 1.2]{Lee-Top-3d}, \cite[Theorem 1.3]{Lee-Top-PIC2}}$] \label{thm:Lee-Top-pinching-exist}
For $n\geq 3$ suppose $(M^n,g_0)$ is a complete non-compact manifold such that
\begin{itemize}
 \item[(i)] for $n=3$ the metric $g_0$ is uniformly Ricci-pinched,
 \item[(ii)] for $n\geq 4$ the metric $g_0$ is uniformly PIC1 pinched.
\end{itemize}
 Then there exists $a > 0$ (depending on the quantitative pinching in (i) resp.~(ii)) and a smooth complete Ricci flow solution $g(t)$ on $M\times [0,\infty)$ with $g(0) = g_0$ and which satisfies $|\Rm(g(t))|_{g(t)}\leq a/t$. Furthermore, 
 \begin{itemize}
 \item[(i)] for $n=3$ the metrics $g(t)$ remain uniformly Ricci-pinched for all $t>0$,
 \item[(ii)] for $n\geq 4$ the metrics $g(t)$ remain uniformly PIC1 pinched for all $t>0$.\footnote{The uniform pinching constant here might be worse for $t>0$.}
\end{itemize}
\end{theo}
 
We emphasize that the existence of a complete Ricci flow solution starting from a non-compact Riemannian manifold with unbounded curvature (under suitable geometric conditions) is a fundamental open problem in the field, with potentially far-reaching applications. Since the proof of the non-collapsing condition (NC) due to Lott carries over to the setting of Theorem \ref{thm:Lee-Top-pinching-exist}
 for $n=3$, Lee-Topping were able to extend Lott's existence (E) and non-collapsing (NC) result to allow for unbounded curvature initially. Combining this with our previous proof of the Hamilton-Lott conjecture (see \cite[Theorem 1.3]{DSS-II}), this yields the Hamilton-Lott conjecture without the assumption of bounded curvature (see \cite[Theorem 1.1]{Lee-Top-3d}). Similarly, combining the existence result of Lee-Topping (Theorem \ref{thm:Lee-Top-pinching-exist}) and Lott's proof of the non-collapsing condition (NC) with 
Theorem \ref{main-II} yields an alternative proof of the Hamilton-Lott conjecture in three dimensions (without the assumption of bounded curvature), which does not directly use the theory of Alexandrov spaces and  RCD spaces.

For $n\geq 4$, under the additional assumption that the initial metric is weakly PIC2 (this is equivalent to non-negative complex sectional curvature, which is preserved under the flow)   Lee-Topping proved  the following pinching theorem. 

\begin{theo}[$\text{\cite[Theorem 1.2]{Lee-Top-PIC2}}$] \label{thm:Lee-Top-PIC2} Suppose $(M^n, g_0)$ is a complete manifold of non-negative complex sectional curvature with $n\geq 3$ that is uniformly PIC1 pinched. Then $(M,g_0)$ is either flat or compact.
\end{theo}
 
Other previous results in higher dimensions under stronger convexity or pinching conditions were obtained by Chen and Zhu \cite{Chen-Zhu}, Ni and Wu \cite{Ni-Wu} and Brendle and Schoen \cite[Theorem 7.1]{Brendle-Schoen09}. 

We note that Theorem \ref{main-II} extends Theorem \ref{thm:Lee-Top-PIC2} in the following way. Although for their existence result (Theorem \ref{thm:Lee-Top-pinching-exist}) Lee-Topping do not need an a priori non-collapsing assumption, the combination of the bound $a/t$ on the curvature of the flow together with the additional weak PIC2 assumption  implies positive asymptotic volume ratio for the flow due to the Gromoll-Meyer injectivity radius estimate, as shown in the proof of \cite[Theorem $4.4$]{Lee-Top-PIC2}.  The non-negativity of the sectional curvature (which is directly implied by the PIC2 assumption) plays a significant role in their proof, and does not follow  if the Riemannian manifold is merely PIC1. Furthermore, the proof of Theorem \ref{thm:Lee-Top-PIC2} uses a differential Harnack inequality proved by  Brendle in \cite{Bre-Har-RF}, which is not known to hold in the PIC1 setting. This is a second instance in the proof of Theorem \ref{thm:Lee-Top-PIC2} where the stronger PIC2 assumption is essential.
As we only  consider the PIC1 setting, this powerful tool is not available to us. A further instance where the PIC2 assumption is important in the proof of Theorem \ref{thm:Lee-Top-PIC2} is in deriving a generalised soul theorem to rule out any non-trivial topology.
 
In the proof of Theorem \ref{main-II} we require an existence result of the type given in (E).  Since we assume that the initial metric is PIC1 pinched, this is provided by Theorem \ref{thm:Lee-Top-pinching-exist}. If one were able to show that this solution is also volume non-collapsed (as is conjectured in \cite{Top-Sur}) this would allow to remove the non-collapsing assumption in Theorem \ref{main-II}.

 \subsection{Outline of paper}

Section \ref{sec-2-rf} collects properties of solutions to the Ricci flow with non-negative Ricci curvature and curvature controlled by $C/t$ close to the initial time. We recall in Section \ref{sec-lo-ell} how to obtain local geometric mollifiers for the distance to the apex on a metric cone in case it is a tangent cone at infinity of a non-collapsed Riemannian manifold with non-negative Ricci curvature. This is  taken from the work of Cheeger-Colding \cite{CheegerColding_Rigidity} and Cheeger-Jiang-Naber \cite{Che-Jia-Nab}. Such mollifiers are constructed along a sequence of times $t_i\searrow 0$ for a solution to  Ricci flow as in the  setting of Theorem \ref{main-I}. 
Section \ref{section-can-fct} then explains how to smooth out the aforementioned mollifiers along such a Ricci flow.
Section \ref{sec-equations} is  devoted to establishing the parabolic equations satisfied by the obstruction tensor $\Ob(t):=\nabla^{g(t),\,2}u-t\Ric(g(t))-\frac{g(t)}{2}$ as well as  the function $v(t):= |\grad u|_g^2(t) -u(t)+t^2\Sc_{g(t)} +2t\, {\rm tr}_{g(t)}\Ob(t)$ 
associated to a solution $u$ of $\partt u = \lap_g u -\frac{n}{2}.$ The tensor $\Ob$ can be seen as  a measurement of  how far away such a solution to the Ricci flow is   from being a self-similar solution: $\Ob$ is zero on an expanding soliton. Using  these equations, Section \ref{section-pt-int-est} proves basic interior estimates for $u$, $v$ and the tensor $\Ob$,  which are then used in Section \ref{sec-int-est} to show 
   faster-than-polynomial decay in the integral sense for $\Ob$: see Proposition \ref{theo-dream-T}. The integral convergence rate for $\Ob$ is then upgraded to pointwise faster-than-polynomial decay in Section \ref{sec-pt-dec-sol} and the section  culminates with  the proof of Theorem \ref{main-I}. Finally, Theorems \ref{main-III} and \ref{main-II} are proved in Section \ref{sec-ric-pin}.

\subsection{Acknowledgements}
The first author is supported by grant ANR-AAPG2020 (Project PARAPLUI) of the French National Research Agency ANR. The third author is  supported by a grant in the Programm  `SPP-2026: Geometry at Infinity' of the German Research Council (DFG).

\section{Basics of Ricci flows with non-negative Ricci curvature}\label{sec-2-rf}
In this section, we collect basic properties of smooth , complete, connected   solutions to the Ricci flow $(M^n,g(t))_{t \in (0,T)}$ satisfying 
\begin{equation}\label{eq:basic_ass}
 \Ric(g(t)) \geq 0,\quad |\Rm(g(t))|_{g(t)} \leq \frac{D_0}{t},\quad \text{for all $t \in (0,T)$.}
\end{equation}

\begin{prop}\label{prop:RF_basics}
Under assumption \eqref{eq:basic_ass} the following statements hold.
\begin{enumerate}
\item There exists $C=C(D_0,n)\in (0,\infty)$ such that for $0< s\leq t<T$ and points $x$ and $y$ in $M$,
\begin{equation}\label{eq:distance-distortion}
d_{g(s)}(x,y)-C\sqrt{t-s}\leq d_{g(t)}(x,y)\leq d_{g(s)}(x,y).\\[1ex]
\end{equation}

\item There is a well defined, unique limiting metric $d_0$ on $M$ as $t\downto 0$,  
\begin{equation*}
d_0(x,y)= \lim_{t\downto 0} d_{g(t)}(x,y),\quad\text{ for all $x,y\in M$.}
\end{equation*}
 Moreover, the metric $d_0$ generates the same topology as that of $(M,d_{g(t)})$ for all $t\in(0,T)$.
We say $(M,g(t))_{t\in (0,T)}$ {\rm is coming out of} $(M,d_0)$.\\
\item If $\AVR(g(t_0))= V_0>0$ for some $t_0 \in [0,T)$ then $\AVR(g(t))= V_0>0$ for all $t \in [0,T)$.
\end{enumerate}
\end{prop}

\begin{proof} The first statement follows from \cite[Section $17$]{HamFor}. The existence of $d_0$ was shown in \cite{SiTo2}. It was further shown there that this implies that the topology of $M$, which agrees with that of $(M,d_{g(t)})$, is also the same as the topology generated by $(M,d_0)$. 

The statement that $\AVR(g(t))$ is constant in time for $t>0$  follows as in  \cite[Theorem $5.2$]{Sch-Sim} (the  proof only requires $\Ric(g(t))\geq 0$ instead of non-negative curvature operator) (cf. Theorem 7, \cite{Yokota}).
 \end{proof}

\section{Local elliptic regularizations of $d_0^2$} \label{sec-lo-ell}

\noindent This short section is devoted to the existence of a sequence of smooth maps that approximates the function $d_0(o,\cdot)^2/4$ on a cone $(C(X),d_0)$ with apex $o$. We state it in our Ricci flow setting but its proof is a mere translation from the results stated in Appendix \ref{appendix-Poisson}. 

\begin{prop}\label{prop-ell-reg}
Let $(M^n,g(t))_{t\in(0,T)}$ be a smooth, complete, connected  Ricci flow such that $
\Ric(g(t))\geq 0$ and $|\Rm(g(t))|_{g(t)}\leq D_0/t$ for $t\in(0,T).$ Assume furthermore that $\AVR(g(t))\geq v>0$, for all $t\in(0,T)$, and that $(C(X),d_0,o):=\lim_{t\rightarrow 0^+}(M^n,g(t),p)$ is a metric cone. Let $\varepsilon_i \searrow 0$. Then there exists $C(n,v) \in (0,\infty)$, $t_i \searrow 0$ and  $u_i \in C^\infty(B_{g(t_i)}(o,4))$ satisfying:\\[-2ex]
\begin{enumerate}
\item (Poisson equation) $\Delta_{g(t_i)}u_i=\frac{n}{2}.$\\[-1ex]
\item (Hessian bounds on $u_i$)
\begin{equation*}
\dashint_{B_{g(t_i)}(o,4)}\left|\nabla^{g(t_i),\,2}u_i-\frac{g(t_i)}{2}\right|^2_{g(t_i)}\,d\mu_{g(t_i)}\leq C(n,v)\varepsilon_i.
\end{equation*}
\item (Sharp $L^2$ gradient bound) 
\begin{equation*}
\dashint_{B_{g(t_i)}(o,4)}\left||\nabla^{g(t_i)}u_i|^2_{g(t_i)}-u_i\right|^2\,d\mu_{g(t_i)}\leq C(n,v)\varepsilon_i.
\end{equation*}
\item ($L^{\infty}$ upper gradient bound) $|\nabla^{g(t_i)}u_i|_{g(t_i)}\leq C(n,v).$\\[-1ex]
\item ($L^{\infty}$ closeness)
\begin{equation}\label{linfty-u-init}
 \sup_{B_{g(t_i)}(o,4)}\left|u_i-\frac{d_{g(t_i)}(o,\cdot)^2}{4}\right|\leq 4 \varepsilon_i.
\end{equation}
\end{enumerate}
\end{prop}

\begin{proof} Note that by Proposition \ref{prop:RF_basics} it holds that $\AVR(g(t)) = V_0>0$ for all $t\in [0,T)$ and thus $\vol_{g(t)}B_{g(t)}(o,r)\geq V_0 r^{-n}>0$ for all $r>0$.  Theorem \ref{che-jia-nab-maps}, applied with $v = V_0/2$, $\varepsilon = \varepsilon_i$, yields $\delta_1 = \delta_{  1}(n,v, \varepsilon)>0$, so that to obtain the desired statements it is sufficient to show that we can choose $ t_i>0$ sufficiently small such that
\begin{itemize}
 \item[(a)] $B_{g(t)}(o, 4\delta^{-1})$ is $(0, \delta^2)$-symmetric.
 \item[(b)] $ \left|\mathcal{W}^{\delta}_{8}(x)-\mathcal{W}^{\delta}_{4}(x)\right| \leq \varepsilon_i $
\end{itemize}
for some $\delta \leq \delta_1,$  for all $t\leq t_i$, 
 where $\mathcal{W}^{\delta}_{s}$ is the local entropy with respect to  $g(t)$, as defined in \cite[Definition 4.19]{Che-Jia-Nab}.

We will first  show that   $(b)$ is satisfied for all $t>0$ sufficiently small, for some $\de \leq \de_1$. This will fix $0<\delta\leq \delta_1$. We will then show that    $(a)$ is satisfied for all $t>0$ sufficiently small, thus completing the proof. 

Note that Theorem \ref{theo-che-jia-nab-vol-entropy}, applied with $\varepsilon':=\varepsilon/3>0$ and $v=V_0/2$ yields $\delta_2>0$, such that if $\delta \leq \delta_0 := \min\{\de_1,\de_2\}$ condition $(\operatorname{b})$, for the smooth metric $g(t)$,  is satisfied by the triangle inequality, provided 
\begin{itemize}
 \item[(c)]  $|\mathcal{V}^0_{g(t)}(o,\sqrt{s}\delta^{-1})-\mathcal{V}^0_{g(t)}(o,\sqrt{s}\delta)|\leq\delta$  for $s=4,8$,
 \item[(d)] \label{item-d} $|\log \mathcal{V}^{\delta^2}_{g(t)}(o,2) - \log \mathcal{V}^{\delta^2}_{g(t)}(o,\sqrt{8})|\leq \varepsilon',$
 \end{itemize}
where 
$\mathcal{V}^\kappa_g(x,r):=\frac{\vol_gB_g(x,r)}{\vol_{-\kappa}\mathbb{B}(r)},$
for any $\kappa \geq 0$ and $\vol_{-\kappa}\mathbb{B}(r)$ is the volume of a ball of radius $r$ in the simply connected space of constant curvature $-\kappa$.
Note that $(M,d_0, o)$ is a cone, and thus $\mathcal{V}^{0}_{d_0}(o,2) = \mathcal{V}^{0}_{d_0}(o,\sqrt{8}) = V_0$, so we can choose $0<\delta\leq \delta_0$ such that
\begin{equation*}
 |\log \mathcal{V}^{\delta^2}_{d_0}(o,2) - \log \mathcal{V}^{\delta^2}_{d_0}(o,\sqrt{8})|\leq \varepsilon'/2\, . 
 \end{equation*}
This fixes $\delta >0$.  Cheeger-Colding's volume continuity Theorem \cite{Cheeger_notes},  then implies that for all $0<t \leq t_i$  sufficiently small
\begin{equation*}
  |\log \mathcal{V}^{\delta^2}_{d_0}(o,2)- \log \mathcal{V}^{\delta^2}_{g(t)}(o,2)| +  |\log \mathcal{V}^{\delta^2}_{d_0}(o,\sqrt{8})- \log \mathcal{V}^{\delta^2}_{g(t)}(o,\sqrt{8})| \leq \varepsilon'/2\, , 
  \end{equation*}
and thus by the triangle inequality condition $(\operatorname{d})$ is satisfied for all such $t>0$.

Using  once again    $\mathcal{V}^0_{d_0}(o,\sqrt{s}\delta^{-1}) = \mathcal{V}^0_{d_0}(o,\sqrt{s}\delta) = V_0$ for $s=4,8$, and Cheeger-Colding's volume continuity Theorem (see \cite{Cheeger_notes}),   we  see that  $(\operatorname{c})$ is satisfied for all $0<t \leq t_i$ sufficiently small, after reducing $t_i$ if necessary.  By the discussion above this yields condition $(\operatorname{b})$ for all $t\leq t_i$ sufficiently small.

Again by by Proposition \ref{prop:RF_basics} we have that $(B_{g(t)}(o, 4\delta^{-1}), d_{g(t)})$ converges in Gromov-Hausdorff distance to $(B_{d_0}(o, 4\delta^{-1}), d_0)$. Since $(M, d_0,o)$ is a metric cone this yields condition $(\operatorname{a})$ for all $0< t\leq t_i$ by further decreasing $t_i$   if necessary.
\end{proof}

\section{Setup and local parabolic regularizations of $d_0^2$}\label{section-can-fct}

\noindent In this section we consider a class of Ricci flows coming out of a cone and outline the strategy of constructing a function mimicking the properties of the (time-dependent) potential function along a self-similarly expanding solution.\\[2ex]
\noindent {\bf Main assumption.} Let $(M^n,g(t))_{t\in(0,T)}$ be a smooth, complete, connected Ricci flow such that 
\begin{equation}
\Ric(g(t))\geq 0,\quad |\Rm(g(t))|_{g(t)}\leq \frac{D_0}{t},\quad t\in(0,T),\tag{Curv}\label{curv-cond}
\end{equation}
 which is uniformly non-collapsed, i.e.~$\AVR(g(t))\geq V_0>0$, for $t\in(0,T)$ and such that the initial condition (IC) satisfies
 \begin{equation}
 \tag{IC}\label{IC}
 (C(X),d_0,o):=\lim_{t\rightarrow 0^+}(M^n,g(t),p),
 \end{equation}
  is a metric cone.\\[2ex]
Recall that a solution $(M^n,g(t))_{t>0}$ is a self-similarly expanding gradient Ricci soliton if $g(t)=t\varphi_t^*g$, $t>0$, where $\partial_t\varphi_t=-t^{-1}\nabla^{g}f\circ\varphi_t$, $\varphi_{t=1}=\Id_M$. Here $f:M\rightarrow\mathbb{R}$ is a smooth function called the soliton potential function. Alternatively, a triple $(M^n,g,\nabla^g f)$ is an expanding gradient Ricci soliton if the following soliton equation holds:
\begin{equation}\label{sol-id}
\Ric(g)-\nabla^{g,2} f=-\frac{g}{2}.
\end{equation}
The soliton identities of such a soliton are:
\begin{equation}\label{id-egs}
|\nabla^g f|^2_g+\R_g=f,\quad \Delta_gf=\R_g+\frac{n}{2},\quad \text{on $M$.}
\end{equation}
The first identity is obtained by applying the second Bianchi identity $2\div_g\Ric(g)=d\,\R_g$ together with the Bochner formula $\div_g\nabla^{g,2}f=d\Delta_gf+\Ric(g)(\nabla^g f,\cdot)$ to the soliton equation \eqref{sol-id} while the second is obtained by tracing the soliton equation \eqref{sol-id}.

For a general solution  coming out of a cone (i.e.~satisfying the main assumption, but not necessarily self-similarly expanding), we aim to mimic the properties of the potential function associated to an expanding gradient Ricci soliton by considering solutions to the following linear heat equation:
 \begin{equation}
\partial_tu=\Delta_{g(t)}u-\frac{n}{2},\quad \text{on $M\times(0,T)$},\quad  u|_{t=0}=\frac{d_0^2(o,\cdot)}{4}.\label{eqn-heat-sing}
\end{equation}
Note that for an asymptotically conical expanding gradient Ricci soliton $(M,g,\nabla^gf)$, we can take $u(\cdot,t)=t\varphi_t^*f$ where again, $\partial_t\varphi_t=-t^{-1}\nabla^gf\circ\varphi_t$. A straightforward computation shows that $$\partial_tu=\Delta_{g(t)}u-\frac{n}{2}=t\R_{g(t)} \ \ \text{and} \ \ \lim_{t\searrow 0} u(t) = \frac{d_0^2(o,\cdot)}{4}\, ,$$
where $o$ is the vertex  and $d_0$ the metric distance of the asymptotic cone: see for instance \cite[Lemma 3.2]{CDS}.

Note that due to the singular initial condition the existence of solutions \eqref{eqn-heat-sing} is not guaranteed by standard methods. Instead we fix a sequence $\varepsilon_i \searrow 0$ and consider the sequence $t_i\searrow 0$ of times $t_i>0$ given by Proposition \ref{prop-ell-reg}, together with functions $u_i \in C^\infty(B_{g(t_i)}(o,4))$ satisfying the estimates $(1) - (5)$. For the ease of notation we consider the sequence of Ricci flows 
\begin{equation}\label{eq:def_sol_RF}
[0, T) \ni t \mapsto g_i(t):= g(t+t_i),
\end{equation}
where we have replaced $T$ by $T-t_1$ to allow for a common time interval of definition for the sequence of flows considered. Note that $d_{g_i(0)} \to d_0$  locally uniformly  as $i \to \infty$  and $g_i (\cdot,\cdot) \to g(\cdot,\cdot) $  as $i\to \infty$ locally smoothly on $M\times (0,T)$. With this set-up we have $u_i\in C^\infty(B_{g_i(0)}(o,4))$ and $u_i$ satisfy the properties $(1) - (5)$ in Proposition \ref{prop-ell-reg} with respect to $g_i(0)$. 

Standard existence theory yields solutions $u^i$ to the following heat equation with Dirichlet boundary conditions
\begin{equation} \label{eqn-heat-sing-bis-rep}
  \left\{
      \begin{aligned}
      \ \partial_tu^i&=\Delta_{g_{i}(t)}u^i-\frac{n}{2},\quad\text{on $B_{g_i(0)}(o,4)\times(0,T)$},\\
        u^i&=u_i,\quad\text{on $ B_{g_i(0)}(o,4)\times \{t=0\}\cup \partial B_{g_i(0)}(o,4)\times (0,T)$}. 
           \end{aligned}
    \right.  
\end{equation}

Observe that the maximum principle applied to $u^i+\frac{n}{2}t$ combined with [\eqref{linfty-u-init}, Proposition \ref{prop-ell-reg}] yields:
\begin{prop}\label{ptwise-bd-ui}
Let $(M^n,g(t))_{t\in (0,T)} $ be a smooth, complete, connected  solution to Ricci flow satsifying the assumptions (Curv) and (IC) of the beginning of this section.  There exists $C=C(n,T) \in (0,\infty)$ such the solutions $(u^i)_{i\in \N}$ to \eqref{eqn-heat-sing-bis-rep}  satisfy \begin{equation*}
\sup_{B_{g_i(0)}(o,4)\times(0,T)}|u^i(x,t)|\leq C(n,T).
\end{equation*}
\end{prop}

We will see in Proposition \ref{prop-grad-bd-a-priori}  that the solutions $u^i$ satisfy uniform interior estimates on higher order covariant derivatives. Furthermore, Corollary \ref{coro-dist-comp-sol} gives uniform control on the attainment of the initial condition. That is, we will see  (up to a subsequence)  that
\begin{equation}\label{eq:def-limit-solution}
 u^i \to u \ \text{locally smoothly on}\ B_{d_0}(o,4) \times (0,T),
\end{equation}
as $i \to \infty$ where $u$ solves 
\begin{equation} \label{eq:eqn-limit-solution}
  \left\{
      \begin{aligned}
      \ \partial_tu&=\Delta_{g(t)}u-\frac{n}{2},\quad\text{on $B_{d_0}(o,4)\times(0,T)$},\\
        u&=\tfrac{d_0^2}{4}\quad\text{on $ B_{d_0}(o,4)\times \{t=0\}\cup \partial B_{d_0}(o,4)\times (0,T)$}.
           \end{aligned}
    \right.
\end{equation}

\section{Equations}\label{sec-equations}
\noindent In this section, we consider a solution to the heat equation
$$\partial_tu=\Delta_{g(t)}u-\frac{n}{2}, $$
 on $U\times(0,T)$ for an open set    $U \subseteq M,$
where $(M,g(t))_{t\in (0,T)}$ is a smooth one parameter family of Riemannian manifolds.
\begin{defn}
Let $(M,g(t))_{t\in (0,T)}$ be a smooth one parameter family of Riemannian manifolds,  
$u:U \times (0,T) \to \mathbb{R}$    smooth, and  $U \subseteq M$ an open set.  
We define the obstruction tensor $ \Ob$ on $U \times (0,T)$ by
\begin{equation*}
\Ob(t):=-t\Ric(g(t))-\frac{g(t)}{2}+\frac{1}{2}\Li_{\nabla^{g(t)}u(t)}(g(t)).
\end{equation*}
\end{defn}
\begin{rk}
If  $(M,g(t))_{t\in (0,T)}$ is a smooth solution to Ricci flow, then  $\Ob(t)=0$ if and only if $(M^n,g(t))_{t>0}$ is a self-similarly expanding gradient Ricci soliton with potential function $u$.
\end{rk}
\begin{prop}\label{prop-eqn-evo-obs-tensor}
 Let $(M,g(t))_{t\in (0,T)}$ be a smooth solution to Ricci flow, $u:U \times (0,T) \to \mathbb{R}$     a smooth solution  to $\partial_tu=\Delta_{g(t)}u-\frac{n}{2}$ on $U \times (0,T)$.  Then, 
 the Hessian $\nabla^{g(t),\,2}u$ satisfies 
 \begin{equation*}
\partial_t\nabla^{g(t),\,2}u=\Delta_{g(t),L}\nabla^{g(t),\,2}u,\quad\text{on $U\times (0,T)$,}
\end{equation*}
where, for any smooth $2$-tensor $S,$ 
$$(\Delta_{g(t),L} S)_{ij} = \Delta_g S_{ij} +2g^{pk}g^{ql}{\Rm(g)}_{ipqj}S_{kl}-g^{kl}\Ric(g)_{ik}S_{jl}-g^{kl}\Ric(g)_{jk}S_{il}$$ 
is the Lichnerowicz Laplacian  of $S$. In particular, the obstruction tensor satisfies the following linear evolution equation
\begin{equation}
\partial_t\Ob=\Delta_{g(t),L}\Ob,\quad\text{on $U\times (0,T)$.}\label{eqn-evo-T}
\end{equation}
Furthermore, 
\begin{equation*}
\partial_t\tr_{g(t)}\Ob=\Delta_{g(t)}\tr_{g(t)}\Ob+2\langle\Ric(g(t)),\Ob(t)\rangle_{g(t)},\quad\text{on $U\times (0,T)$.}
\end{equation*}
\end{prop}

\begin{proof}
A computational proof of the first identity can be found for instance in \cite[Lemma $2.33$]{CLN}: indeed, one can show that for any time-dependent smooth function $u$, 
\begin{equation*}
\left(\partial_t-\Delta_{g(t),L}\right)\nabla^{g(t),\,2}u=\nabla^{g(t),\,2}\left(\partial_tu-\Delta_{g(t)}u\right).
\end{equation*}
This proof relies on the full Bianchi identity together with commutation formulae involving the curvature and its covariant derivatives.

Then \eqref{eqn-evo-T} follows by noting that the tensor $g(t)+2t\Ric(g(t))$ satisfies the same heat equation as $\nabla^{g(t),\,2}u$ does since $\partial_t\Ric(g(t))=\Delta_{g(t),L}\Ric(g(t))$ and $\partial_tg(t)=-2\Ric(g(t))$ by definition of the Ricci flow.

The last statement is obtained by tracing \eqref{eqn-evo-T} with respect to $g(t)$ together with the fact that $\partial_t\tr_{g(t)}S(t)=\tr_{g(t)}(\partial_tS(t))+2\langle \Ric(g(t)),S(t)\rangle_{g(t)}$ for an arbitrary time-dependent smooth symmetric $2$-tensor $S(t)$.
\end{proof}

\begin{defn}\label{defn-v}
Let $(M,g(t))_{t\in (0,T)}$ be a smooth one parameter family of Riemannian manifolds, und $u: U \times(0,T) \to \mathbb{R}$ a smooth function. 
We define for $t\in(0,T)$,
\begin{equation*}
v(t):=|\nabla^{g(t)}u|^2_{g(t)}-u+t^2\R_{g(t)}+2t\tr_{g(t)}\Ob(t).
\end{equation*}
\end{defn}

\begin{rk}
Each quantity $|\nabla^{g(t)}u|_{g(t)}^2-u+t^2\R_{g(t)}$ and $\tr_{g(t)}\Ob(t)$ vanishes on an expanding gradient Ricci soliton with soliton potential function $u/t,$ in the case that  $(M,g(t))_{t\in (0,T)}$ is a smooth Ricci flow. 
\end{rk}

\begin{prop}\label{prop-evo-eqn-v}
 Let $(M,g(t))_{t\in (0,T)}$ be a smooth solution to Ricci flow, and $u: U \times (0,T) \to \mathbb{R}$ a smooth solution to $\partt u = \lap_{g(t)} u -\frac n 2$. The function $v(t)$ satisfies the following non-homogeneous heat equation along the Ricci flow: 
\begin{equation*}
(\partial_t-\Delta_{g(t)})v(t)=-2|\Ob(t)|^2_{g(t)}.
\end{equation*}

\end{prop}

\begin{proof}
The difference $|\nabla^{g(t)}u|^2_{g(t)}-u$ satisfies:
\begin{equation*}
\begin{split}
\left(\partial_t-\Delta_{g(t)}\right)\left(|\nabla^{g(t)}u|^2_{g(t)}-u\right)&=-2\left|\nabla^{g(t),\,2}u-\frac{g(t)}{2}\right|^2_{g(t)}-2\partial_tu\\
&=-2\left|\nabla^{g(t),\,2}u-\frac{g(t)}{2}\right|^2_{g(t)}-2\tr_{g(t)}\left(\nabla^{g(t),\,2}u-\frac{g(t)}{2}\right).
\end{split}
\end{equation*}
By Proposition \ref{prop-eqn-evo-obs-tensor}, 
\begin{equation*}
\begin{split}
&\left(\partial_t-\Delta_{g(t)}\right)\left(|\nabla^{g(t)}u|^2_{g(t)}-u+2t\tr_{g(t)}\Ob(t)\right)\\
&=-2\left|\nabla^{g(t),\,2}u-\frac{g(t)}{2}\right|^2_{g(t)}-2\tr_{g(t)}\left(\nabla^{g(t),\,2}u-\frac{g(t)}{2}\right)\\
&\quad+2\tr_{g(t)}\Ob+4t\langle\Ric(g(t)),\Ob(t)\rangle_{g(t)}\\
&=-2\left|\Ob(t)+t\Ric(g(t))\right|^2_{g(t)}+4t\langle\Ric(g(t)),\Ob(t)\rangle_{g(t)}-2t\R_{g(t)}\\
&=-2|\Ob(t)|^2_{g(t)}-2t^2|\Ric(g(t))|_{g(t)}^2-2t\R_{g(t)},
\end{split}
\end{equation*}
by the very definition of $\Ob$. Since $(\partial_t-\Delta_{g(t)})\R_{g(t)}=2|\Ric(g(t))|^2_{g(t)}$, the expected computation follows.
\end{proof}

Here is a useful formula which mimics the corresponding computation on an  expanding gradient Ricci soliton and which does not depend on the equation satisfied by the solution $u(t)$:
\begin{lemma}\label{lemma-gen-form-exp-ric-pin}
Let $(M^n,g(t))_{t\in(0,T)}$ be a smooth one parameter family of Riemannian manifolds  and let $u:U\times(0,T)\rightarrow \mathbb{R}$ be a smooth  function and let $\Ob: M \times (0,T) \rightarrow  \mathbb{R}$ be the corresponding obstruction tensor.   Then the following holds true pointwise on $U$:
\begin{equation}
d\,(t\R_{g(t)})+2\Ric(g(t))(\nabla^{g(t)}u(t),\,\cdot\,)=2\left(\div_{g(t)}-d\tr_{g(t)}\right)\Ob(t).\label{beautiful-covid19}
\end{equation}
Moreover,  
\begin{equation}
  \begin{split}
  d(v(t))& = d\left(t^2\R_{g(t)}+|\nabla^{g(t)}u|^2_{g(t)}-u(t)+2t\tr_{g(t)}\Ob(t)\right)\\
& =2\left(t\div_{g(t)}\Ob(t)+\Ob(t)(\nabla^{g(t)}u(t),\,\cdot\,)\right).\label{beautiful-covid19-bis}
\end{split}
\end{equation}
\end{lemma}
\begin{rk}
The advantage of Proposition \ref{prop-evo-eqn-v} over [\eqref{beautiful-covid19-bis}, Lemma \ref{lemma-gen-form-exp-ric-pin}] is that the function $\R_g+|\nabla^gu|^2_g-u+2\tr_g\Ob$ depends on $\Ob$ in an integral sense by Duhamel's principle. On the other hand, [\eqref{beautiful-covid19-bis}, Lemma \ref{lemma-gen-form-exp-ric-pin}] involves space derivatives only compared to Proposition \ref{prop-evo-eqn-v}.
\end{rk}
\begin{proof}
By using the traced version of the Bianchi identity, we compute as follows:
\begin{equation*}
\begin{split}
\left(\div_{g(t)}-d\tr_{g(t)}\right)\left(t\Ric(g(t))+\frac{g(t)}{2}\right)=-\frac{t}{2}d\,\R_{g(t)},
\end{split}
\end{equation*}
where we have used the fact that $g(t)$ is parallel.
Now, 
\begin{equation*}
\left(\div_{g(t)}-\frac{1}{2}d\tr_{g(t)}\right)\Li_{\nabla^{g(t)}u(t)}(g(t))= d\left(\Delta_{g(t)}u(t)\right)+2\Ric(g(t))(\nabla^{g(t)}u(t),\,\cdot\,),
\end{equation*}
where we have used the Bochner formula 
\begin{equation*}
\div_{g(t)}\nabla^{g(t),\,2}u(t)=d(\Delta_{g(t)}u(t))+\Ric(g(t))(\nabla^{g(t)} u(t),\,\cdot\,),
\end{equation*}
to commute the derivatives. In particular we conclude that
\begin{equation*}
\begin{split}
&\left(\div_{g(t)}-d\tr_{g(t)}\right) \Ob(t) =\left(\div_{g(t)}-d\tr_{g(t)}\right)\left(-t\Ric(g(t))-\frac{g(t)}{2}+\frac{1}{2}\Li_{\nabla^{g(t)}u(t)}(g(t))\right)=\\
&\frac{t}{2}d\,\R_{g(t)}+\frac{1}{2}\left(d\left(\Delta_{g(t)}u(t)\right)+2\Ric(g(t))(\nabla^{g(t)}u(t),\,\cdot\,)\right)-\frac{1}{4}d\left(\tr_{g(t)}\Li_{\nabla^{g(t)}u(t)}(g(t))\right)  =\\
\,&\frac{t}{2}d\,\R_{g(t)}+\Ric(g(t))(\nabla^{g(t)}u(t),\,\cdot\,),
\end{split}
\end{equation*}
as expected. 

To prove (\ref{beautiful-covid19-bis}), observe that 
\begin{equation*}
\begin{split}
2\Ob(t)(\nabla^{g(t)}u(t),\,\cdot\,)=\,&(\Li_{\nabla^{g(t)}u(t)}(g(t))-2t\Ric(g(t))-g(t))(\nabla^{g(t)}u(t),\,\cdot\,)\\
=\,&d\left(|\nabla^{g(t)}u(t)|^2_{g(t)}-u(t)\right)-2t\Ric(g(t))(\nabla^{g(t)}u(t),\,\cdot\,),
\end{split}
\end{equation*}
by the very definition of $\Ob(t)$. Then \eqref{beautiful-covid19} together with the previous observation leads to the proof of \eqref{beautiful-covid19-bis}.
\end{proof}

\section{Pointwise interior estimates}\label{section-pt-int-est}

\noindent In section \ref{sec-pt-dec-sol} we will  take a limit of the sequence of solutions $(u^i)_{i\in \N}$ to \eqref{eqn-heat-sing-bis-rep} and Ricci flows $g_i(t)$ from Section \ref{section-can-fct}.
To obtain a limit, and to study the limit effectively, it will be necessary 
to prove estimates for $u^i$ and $g_i$. This is the content of the next two sections.
The results are  presented in a more general setting, but ultimately (in Section \ref{sec-pt-dec-sol})   the solutions $g(t)$   and $u(t)$ and $u(0)$ of this section will correspond to a solution  $g_i(t)$ and $u^i(t)$, with $u^i(0) =u_i$ from section \ref{section-can-fct}, where $g(t)$ is a solution coming out of a   
cone as in Theorem \ref{main-I}.  Let us start with standard interior estimates on the covariant derivatives of $u$:
\begin{prop}\label{prop-grad-bd-a-priori}
Let $(M,g(t))_{t\in [0,T)}$ be a smooth, complete, connected solution 
to Ricci flow, satisfying the basic assumption  \eqref{eq:basic_ass}   and $u:B_{g_0}(o,4) \times [0,T) \to \mathbb{R}$ a smooth solution to $\partt u = \lap_{g(t)} u -\frac n 2$ satisfying  
$$\sup_{B_{g_0}(o,4)}|\grad u_0|_{g_0}^2  +  \sup_{B_{g_0}(o,4) \times [0,T) } |u|^2 \leq A.$$
Then, for all $r_0\in(0,4)$,  for all $k\in \N_0,$ there exists $C_k=C(n,k,r_0,D_0,A) \in (0,\infty)$ such that  
the covariant derivatives of $u$  satisfy 
\begin{equation}\label{loc-upper-bd-grad-u}
\begin{split}
&|\nabla^{g(t)} u(t)|_{g(t)} \leq C_1,  \ \ \text{on $B_{g_0}(o,r_0)\times (0,T)$,} \\
&|\nabla^{g(t),\,k}u(t)|_{g(t)}\leq \frac{C_k}{t^{\frac{k-1}{2}}}  \ \  \text{on $B_{g_0}(o,r_0)\times (0,T)$, for $k\geq 2$, }
\end{split}
\end{equation}
and the covariant derivatives of the obstruction tensor satisfy 
\begin{equation}\label{loc-upper-bd-grad-ob}
|\nabla^{g(t),\,k}\Ob(t)|_{g(t)}\leq \frac{C_k}{t^{\frac{k+1}{2}}},\quad\text{on $B_{g_0}(o,r_0)\times (0,T)$ for $k\in \N_0$.}
\end{equation}
\end{prop}
\begin{proof}
Since the solutions $u^i$ are uniformly locally bounded in space and time thanks to Proposition \ref{ptwise-bd-ui}, the proof of \eqref{loc-upper-bd-grad-u} is standard: see for instance the proofs of \cite[Theorems $2.1$ and $2.2$]{DSS-1} for $k=1$ and $k=2$. The proofs for $k\geq 3$ follow analogously. The power $(k-1)/2$ comes from the fact that $|\nabla^{g(t)}u(t)|_{g(t)}$ is locally bounded in space uniformly in time.

The proof of \eqref{loc-upper-bd-grad-ob} uses the previously established interior bounds on $\nabla^{g(t),k}u(t)$ together with Shi's interior bounds on the curvature $|\nabla^{g(t),k}\Ric(g(t))|_{g(t)}\leq C_kt^{-(k+2)/2}$ for $k\geq 0$ and the definition of $\Ob$ from Section \ref{sec-equations}.   
 \end{proof}

We are now in a position to state faster than polynomial decay on the $L^2_{loc}$ norm of $v_+$.
\begin{lemma}[Upper bound on $v$]\label{lemma-L2-v-plus}
Under the same assumptions as  Proposition \ref{prop-grad-bd-a-priori} with the further assumption 
that $$\int_{B_{g_0}(o,3)} v^2(0)\,d\mu_{g_0} \leq \varepsilon \leq 1,$$ and  for $B_{g_0}(p,2r)\subset B_{g_0}(o,3)$ and each $k\geq 1$, there exists $C_k=C(n,k,r,D_0,A) \in (0,\infty)$ such that for $t \in (0,\min\{T,B(n) 2^{-2k}r^2/D_0^2\})$ where $ B= B(n)\in (0,\infty) $ is a uniform positive constant,
\begin{equation*}
\int_{B_{g_0}(p,2^{-k}r)}v_+^2(t)\,d\mu_{g(t)}\leq \int_{B_{g(t)}(p,2^{-k}r)}v_+^2(t)\,d\mu_{g(t)} \leq C_k(t^{k}+\varepsilon).
\end{equation*}

\end{lemma}

\begin{proof}
According to Proposition \ref{prop-evo-eqn-v}, $\left(\partial_t-\Delta_{g(t)}\right)v\leq 0$ on $B_{g_0}(o,3)\times[0,T)$. In particular, $v_+^2$ satisfies in the weak sense $(\partial_t-\Delta_{g(t)})v_+^2\leq -2|\nabla^{g(t)}v_+|_{g(t)}^2$ on $B_{g_0}(o,3)\times[0,T)$. Then multiplying across this inequality by $\eta$ where $\eta$ is a Perelman type cut-off function with respect to $p$ and $r_1:=\alpha r$ and $r_2:=\beta r$ (see Lemma \ref{lemma-Cutoff-Perelman}) with $0<\alpha<\beta \leq 1$, an integration by parts gives:
\begin{equation}
\begin{split}\label{alpha-beta-gal-ineq}
\int_M\eta v_+^2\,d\mu_{g(s)}\bigg\vert_{0}^t&= \int_0^t\int_M\left(\partial_s-\R_{g(s)}\right)\eta v_+^2\,d\mu_{g(s)}ds \\
&\leq \int_0^t\int_M\left(\Delta_{g(s)}\eta\right)v_+^2+\eta\partial_sv_+^2\,d\mu_{g(s)}ds \\
&\leq\int_0^t\int_M-2g(s)(\nabla^{g(s)}\eta,\nabla^{g(s)}v_+) v_++\eta\partial_sv_+^2\,d\mu_{g(s)}ds \\
&\leq\int_0^t\int_M-2\eta|\nabla^{g(s)}v_+|^2_{g(s)}-4g(s)(\nabla^{g(s)}\eta,\nabla^{g(s)}v_+) v_+\,d\mu_{g(s)}ds\\
&\leq 2\int_0^t\int_M\frac{|\nabla^{g(s)}\eta|^2_{g(s)}}{\eta}v_+^2\,d\mu_{g(s)}ds\\
&\leq C(n,\alpha,\beta,r)\int_0^t\int_{B_{g(s)}(p,\beta r)}v_+^2\,d\mu_{g(s)}ds.
\end{split}
\end{equation}
In particular, if $\alpha=1/2$ and $\beta=1$ then:
\begin{equation*}
\begin{split}\label{feed-me-v}
\int_{B_{g_0}(p,r/2)}v_+^2\,d\mu_{g(t)}&\leq\int_{B_{g(t)}(p,r/2)}v_+^2\,d\mu_{g(t)}\leq 
C(n,r) \int_M\eta v_+^2\,d\mu_{g(t)}\\
& \leq C(n,r)\int_0^t\int_{B_{g(s)}(p,r)}v_+^2\,d\mu_{g(s)}ds+ C(n,r) \varepsilon\\
&\leq C(n,r)\int_0^t\int_{B_{g_0}(p,r+C\sqrt{s})}v_+^2\,d\mu_{g(s)}ds+ C(n,r) \varepsilon\\
&\leq C(n,r)\int_0^t\int_{B_{g_0}(o,3)}v_+^2\,d\mu_{g(s)}ds+C(n,r) \varepsilon.
\end{split}
\end{equation*}
Here we have used $\int_{B_{g_0}(o,3)} v^2(0)\,d\mu_{g_0} \leq \varepsilon \leq 1$ in the second line by assumption and we have been using that $B_{g_0}(p,r/2)\subset B_{g(s)}(p,r/2)$ in the second inequality together with $B_{g(s)}(p,r)\subset B_{g_0}(p,r+C\sqrt{s})$ according to Proposition \ref{prop:RF_basics} in the penultimate line.

Since $v$ is locally uniformly bounded in time thanks to the interior estimates from Proposition \ref{prop-grad-bd-a-priori},  one gets:
\begin{equation}
\begin{split}\label{can-do-better-v}
\int_{B_{g_0}(p,r/2)}v_+^2\,d\mu_{g(t)}\leq \int_{B_{g(t)}(p,r/2)}v_+^2\,d\mu_{g(t)}\leq C_1(t+\varepsilon),
\end{split}
\end{equation}
for some $C_1=C_1(n,r,D_0,A) \in (0,\infty)$.

Applying \eqref{alpha-beta-gal-ineq} to $\alpha=1/4$ and $\beta=1/2$ in lieu of $\alpha=1/2$ and $\beta=1$ and using the previous bound \eqref{can-do-better-v} lead to the existence of  $C_2=C_2(n,r,D_0,A) \in (0,\infty)$ such that:
\begin{equation*}
\begin{split}
\int_{B_{g_0}(p,r/4)}v_+^2\,d\mu_{g(t)}&\leq \int_{B_{g(t)}(p,r/4)}v_+^2\,d\mu_{g(t)}\leq C(n,r/2)(\int_0^t\int_{B_{g(s)}(p,r/2)}v_+^2\,d\mu_{g(s)}ds+\varepsilon)\\
&\leq C(n,r/2)(\int_0^tC_1(s+\varepsilon)\,ds+ \varepsilon)\\
& \leq \frac{1}{2}C(n,r/2)C_1t^2+C(n,r/2)\varepsilon t+C(n,r/2)\varepsilon\leq C_2(t^2+\varepsilon).
\end{split}
\end{equation*}
 Here we have used $\varepsilon\leq 1$ in the last inequality.
By induction on $k$, one gets the expected bound on $v_+$.
\end{proof}

The following corollary turns the $L^2$ bound of $v_+$ into a pointwise bound.

\begin{coro}\label{coro-sup-v_+}
Under the same assumptions as   Proposition \ref{prop-grad-bd-a-priori} with the further assumption  
 $$\int_{B_{g_0}(o,3)} v^2(0)\,d\mu_{g_0} \leq \varepsilon \leq 1,$$ and  for each $k\geq 1$, there exists $C_k=C(n,k,r,D_0,A) \in (0,\infty)$ such that if $B_{g_0}(p,\sqrt{t})\times (2^{-1}t,2t)\subset B_{g_0}(p, 2^{-k}r)\times (0,\min\{T,B(n)2^{-2k}r^2/D_0^2\})\subset B_{g_0}(o, 3)\times (0,T)$ for some uniform positive constant $B(n) \in (0,\infty)$ then:
\begin{equation*}
v_+^2(p,t)\leq C_kt^{-n/2}\left(t^{k}+\varepsilon\right).
\end{equation*}
\end{coro}

\begin{proof}
Recall from  Proposition \ref{prop-evo-eqn-v} that $v$ satisfies $(\partial_t-\Delta_{g(t)})v\leq 0$, i.e.~$v$ is a subsolution to the heat equation along the Ricci flow $(g(t))_t$. Choose $k>n$ and perform a local Nash-Moser  iteration on each ball $B_{g_0}(p,\sqrt{t})\times\left(t,2t\right)\subset B_{g_0}(p,2^{-k}r)\times(0,\min\{T,B(n)2^{-2k}r^2/D_0^2\})$ to get for each $\theta\in(0,1)$,  
\begin{equation}\label{loc-nash-moser}
\sup_{B_{g_0}(p,\sqrt{\theta t})\times\left(t(1+\theta),\,2t\right)}v_+^2\leq C(n,\theta,D_0)\, \dashint_t^{2t}\dashint_{B_{g_0}(p,\sqrt{t})}v_+^2\,d\mu_{g(s)}ds.
\end{equation} 
 See for instance 
 %\cite[Theorem $6.17$]{Boo-Lieberman}
  \cite{Saloff-Coste-Riem}, \cite{Grig92} or \cite[Theorem 25.2]{Chow2010TheRF}
  % \cite[Chapter 12]{DiBenedetto_Gianazza}, 
  for a proof.

Now apply Lemma \ref{lemma-L2-v-plus} so that if $t<B2^{-2k}r^2D_0^{-2}$, the previous inequality \eqref{loc-nash-moser} leads to the pointwise bound:
\begin{equation*}
\begin{split}
\sup_{B_{g_0}(p,\sqrt{\theta t})\times\left(t(1+\theta),\,2t\right)}v_+^2&\leq C(n,k,\theta,D_0)\big(\tfrac{r^2}{t}\big)^{\frac{n}{2}}\dashint_t^{2t}\dashint_{B_{g_0}(p,2^{-k} r)}v_+^2\,d\mu_{g(s)}ds\\
&\leq C(n,k,\theta,D_0)t^{-\frac{n}{2}}\dashint_t^{2t}s^k+\varepsilon\,ds\\
&\leq  C(n,k,\theta,D_0) t^{-n/2} \Big( t^k+\varepsilon\Big),
\end{split}
\end{equation*}
 as expected.
\end{proof}

\section{Integral estimates}\label{sec-int-est}
In this section we partially  consider the   setting  of  Section \ref{section-pt-int-est} with the notation of Section \ref{sec-equations}. Nevertheless the setting of Lemma \ref{magic-comput} is more general and the object $\Ob$ is not necessarily the tensor defined in Section \ref{sec-equations}. We also consider both settings combined: in this case we always assume the notation of Section  \ref{sec-equations}.  The setting being considered is detailed in the statement of each proposition, lemma, corollary or remark.   The reader not interested in these more general settings may assume the setting of Theorem \ref{main-I} which implies both the setting of Lemma \ref{magic-comput} and Section \ref{section-pt-int-est}. 
\begin{lemma}[Rough energy bound on $\Ob$]\label{coro-rough-T}
We consider  the same assumptions as in Proposition \ref{prop-grad-bd-a-priori}. Then there exist $C'(n,A,D_0) \in (0,\infty)$ such that:
\begin{equation*}
\begin{split}
\int_0^t\int_{B_{g_0}(o,3)}|\Ob(s)|^2_{g(s)}\,d\mu_{g(s)}ds&\leq C',\quad t\in(0, \min\{T,B(n)/D_0^2\}).
\end{split}
\end{equation*}
\end{lemma}

\begin{proof} First note that 
\begin{equation*}
|\Ob(t)|^2_{g(t)}\leq C(n)\left(|\nabla^{g(t),\,2}u(t)|^2_{g(t)}+|g(t)+2t|\Ric(g(t))|^2_{g(t)}\right) \leq C(n)  \left(1+ |\nabla^{g(t),\,2}u(t)|^2_{g(t)}\right),
\end{equation*}
for some uniform $C(n)\in (0,\infty)$. Here we used that $t|\Ric(g(t))|_{g(t)}$ is bounded by assumption \eqref{eq:basic_ass}. It is thus sufficient to show that
\begin{equation*}
 \int_0^t\int_{B_{g_0}(o,3)}|\nabla^{g(t),\,2}u(t)|^2_{g(t)}\,d\mu_{g(s)}ds\leq C',\quad t\in(0, \min\{T,B(n)/D_0^2\}). 
\end{equation*}
Let $\eta$ be a Perelman type cut-off function  with respect to the point $o$ and radii $r_1:=3$ and $r_2:=7/2$ (see Lemma \ref{lemma-Cutoff-Perelman}).
The Bochner formula for functions $\Delta_{g(t)}\nabla^{g(t)}u(t)=\Delta_{g(t)}\nabla^{g(t)}u(t)+\Ric(g(t))(\nabla^{g(t)}u(t))$ implies that
\begin{equation*}
\left(\partial_t-\Delta_{g(t)}\right)|\nabla^{g(t)}u|_{g(t)}^2=-2|\nabla^{g(t),\,2}u|^2_{g(t)}.
\end{equation*}
 Therefore, due to the properties of a Perelman type cut-off function,
\begin{equation*}
\begin{split}
\frac{d}{dt}& \int_{M} |\nabla^{g(t)} u|_{g(t)}^2 \eta\,d\mu_{g(t)} \leq   \\
&- \int_{M}2|\grad^{g(t),\,2} u(t) |_{g(t)}^2 \eta + 2g(t)(\grad^{g(t)} \eta, \grad^{g(t)} |\grad^{g(t)} u(t)|_{g(t)}^2) + \R_{g(t)} |\grad^{g(t)} u(t)|^2\eta\,d\mu_{g(t)} \\
& \leq  \int_{M}  -2|\grad^{g(t),\,2} u(t)|_{g(t)}^2 \eta  + c(n)|\grad^{g(t)} \eta|_{g(t)} |\grad^{g(t)} u(t)|_{g(t)}|\grad^{g(t),\,2} u(t)|_{g(t)}\,d\mu_{g(t)}\\
& \leq  \int_{M}- |\grad^{g(t),\,2} u(t)|_{g(t)}^2 \eta  +  \frac{|\grad^{g(t)} \eta|_{g(t)}^2}{\eta} |\grad^{g(t)} u(t)|_{g(t)}^2\,d\mu_{g(t)} \\
& \leq -\int_{B_{g_0}(o,3)} |\grad^{g(t),\,2} u(t)|_{g(t)}^2\,d\mu_{g(t)}  
+ C(n)\int_{B_{g(t)}(o,7/2)}   |\grad^{g(t)} u(t)|^2 \,d\mu_{g(t)}\\
&\leq -\int_{B_{g_0}(o,3)} |\grad^{g(t),\,2} u(t)|_{g(t)}^2\,d\mu_{g(t)}  
+ C(n)\int_{B_{g_0}(o,7/2+C\sqrt{t})}   |\grad^{g(t)} u(t)|^2 \,d\mu_{g(t)}\\
&\leq -\int_{B_{g_0}(o,3)} |\grad^{g(t),\,2} u(t)|_{g(t)}^2\,d\mu_{g(t)}  
+ C(n)\int_{B_{g_0}(o,15/4)}   |\grad^{g(t)} u(t)|^2 \,d\mu_{g(t)}.
\end{split}
\end{equation*}
Here we have been using that $B_{g(t)}(p,7/2)\subset B_{g_0}(p,7/2+C\sqrt{t})$ according to Proposition \ref{prop:RF_basics} in the penultimate line and the fact that $7/2+C\sqrt{t}\leq 15/4<4$ by choice on $t$. We have also used $B_{g_0}(p,3)\subset B_{g(t)}(p,3)$ in the antepenultimate line.

Integrating in time and using [\eqref{loc-upper-bd-grad-u}, Proposition \ref{prop-grad-bd-a-priori}] with $k=1$ gives the result.
\end{proof}
The next result uses the rescaled scalar curvature $t\R_{g(t)}$ as a barrier to estimate the tensor $\Ob(t).$ Recall that for two tensors $A$ and $B$ on $M$, we denote by $A\ast_{g} B$ any linear combination of tensorial contractions of $A$ and $B$ with respect to a Riemannian metric $g$ on $M$.

\begin{lemma}\label{magic-comput}
Let $(U,g(t))_{t\in[0,T)}$ be a solution to Ricci flow, not necessarily complete, satisfying the following:
\begin{enumerate}
\item  \label{ass-1-magic} There exists $K\in (0,\infty)$ such that for all $t\in[0,T)$, $|\Rm(g(t))|_{g(t)}\leq K\,\R_{g(t)}$ on $U$.
\item \label{ass-2-magic} There exists $c(n)\in (0,\infty)$ such that $t\R_{g(t)}\leq \frac{c(n)}{ K},$ for all $t\in (0,T) $  on $U$.
\end{enumerate}
 Let $\Ob(t)$ be a time-dependent tensor satisfying
\begin{equation*}
\frac{\partial}{\partial t}\Ob(t)=\Delta_{g(t)}\Ob(t)+\Rm(g(t))\ast_{g(t)} \Ob(t),\quad \text{on $U\times(0,T)$.}
\end{equation*}

 Then   there exists $\beta=\beta(n,K)\in (0,\infty)$ such that
 \begin{equation}
\begin{split}\label{est-T-heat-subsol}
\left(\partial_t-\Delta_{g(t)}\right)\left(\frac{|\Ob(t)|^2_{g(t)}}{(1+\beta t\R_{g(t)})^{1/6}}\right)&\leq  -\frac{1}{3^{1/6}}|\nabla^{g(t)}\Ob(t)|^2_{g(t)}-\frac{1}{12}\frac{|\nabla^{g(t)}(1+\beta t\R_{g(t)})|^2_{g(t)}}{(1+\beta t\R_{g(t)})^{1/6+2}}|\Ob(t)|^2_{g(t)}\\
&\quad-2c(n)K\R_{g(t)}\frac{|\Ob(t)|^2_{g(t)}}{(1+\beta t\R_{g(t)})^{1/6+1}}.
\end{split}
\end{equation}
\end{lemma}

\begin{rk} 
In Lemma \ref{magic-comput}, if $|\Rm(g(t))|_{g(t)}\leq K$ on $U\times(0,T)$ for some $K\in (0,\infty)$ instead of assumptions \eqref{ass-1-magic},   then 
\begin{equation}\label{replaceeq}  
\begin{split}
\left(\partial_t-\Delta_{g(t)}\right)\left(e^{-Ct}|\Ob(t)|^2_{g(t)}\right)&\leq  -2e^{-Ct}|\nabla^{g(t)}\Ob(t)|^2_{g(t)},
\end{split}
\end{equation}
on $U$ 
for a large constant $C=C(n,K)\in (0,\infty)$, in view of the equation 
 \begin{equation*}
\left(\partial_t-\Delta_{g(t)}\right) \Ob(t) =  \Rm(g(t)) \ast_{g(t)}  \Ob(t),
\end{equation*}
satisfied by $\Ob(t)$. 
\end{rk}
\begin{proof}
For an arbitrary positive smooth function $V$,  
\begin{equation}
\begin{split}\label{first-rough-inequ}
\left(\partial_t-\Delta_{g(t)}\right)\left(\frac{|\Ob|^2_{g(t)}}{V^{\alpha}}\right)&=\frac{1}{V^{\alpha}}\left(\partial_t-\Delta_{g(t)}\right)|\Ob|^2_{g(t)}-2g(t)(\nabla^{g(t)}|\Ob|^2_{g(t)},\nabla^{g(t)}V^{-\alpha})\\
&\quad+|\Ob|^2_{g(t)}\left(\partial_t-\Delta_{g(t)}\right)V^{-\alpha}\\
&=\frac{1}{V^{\alpha}}\left(-2|\nabla^{g(t)}\Ob|^2_{g(t)}+\Rm(g(t))\ast_{g(t)} \Ob\ast_{g(t)} \Ob\right)\\
&\quad+\frac{2\alpha}{V^{\alpha+1}}g(t)(\nabla^{g(t)}|\Ob|^2_{g(t)},\nabla^{g(t)}V)\\
&\quad-\alpha|\Ob|^2_{g(t)}\left(\frac{\left(\partial_t-\Delta_{g(t)}\right)V}{V^{\alpha+1}}+(\alpha+1)\frac{|\nabla^{g(t)}V|^2_{g(t)}}{V^{\alpha+2}}\right)\\
&\leq \frac{1}{V^{\alpha}}\left(-2|\nabla^{g(t)}\Ob|^2_{g(t)}+c(n)|\Rm(g(t))|_{g(t)}|\Ob|^2_{g(t)}\right)\\
&\quad+\frac{2\alpha}{V^{\alpha+1}}g(t)(\nabla^{g(t)}|\Ob|^2_{g(t)},\nabla^{g(t)}V)\\
&\quad-\alpha|\Ob|^2_{g(t)}\left(\frac{\left(\partial_t-\Delta_{g(t)}\right)V}{V^{\alpha+1}}+(\alpha+1)\frac{|\nabla^{g(t)}V|^2_{g(t)}}{V^{\alpha+2}}\right).
\end{split}
\end{equation}
Now, 
\begin{equation*}
\begin{split}
\frac{2\alpha}{V^{\alpha+1}}g(t)(\nabla^{g(t)}|\Ob|^2_{g(t)},\nabla^{g(t)}V)\leq \frac{1}{V^{\alpha}}|\nabla^{g(t)}\Ob|^2_{g(t)}+4\alpha^2\frac{|\nabla^{g(t)}V|^2_{g(t)}|\Ob|^2_{g(t)}}{V^{\alpha+2}},
\end{split}
\end{equation*}
which implies in combination with \eqref{first-rough-inequ} 
\begin{equation}
\begin{split}\label{sec-rough-inequ}
\left(\partial_t-\Delta_{g(t)}\right)\left(\frac{|\Ob|^2_{g(t)}}{V^{\alpha}}\right)&\leq -\frac{1}{V^{\alpha}}|\nabla^{g(t)}\Ob|^2_{g(t)}+c(n)|\Rm(g(t))|_{g(t)}\frac{|\Ob|^2_{g(t)}}{V^{\alpha}}\\
&\quad+\alpha(3\alpha-1)\frac{|\nabla^{g(t)}V|^2_{g(t)}}{V^{\alpha+2}}|\Ob|^2_{g(t)}-\alpha |\Ob|^2_{g(t)}\frac{\left(\partial_t-\Delta_{g(t)}\right)V}{V^{\alpha+1}}.
\end{split}
\end{equation}

For $V:=1+\beta t\R_{g(t)}$ with $\beta >0$ to be chosen later, we have $V\geq 1$ since $\R_{g(t)}\geq 0$ and 
\begin{equation*}
\left(\partial_t-\Delta_{g(t)}\right)V=\beta\left(\R_{g(t)}+2t|\Ric(g(t))|^2_{g(t)}\right)\geq \beta \R_{g(t)},
\end{equation*}
according to the evolution equation satisfied by the scalar curvature along a solution to the Ricci flow. 
We define $ \varepsilon:= \frac{1}{c(n)K 12}$ so that $t\R_{g(t)} \leq \varepsilon.$
Then, if $\alpha:=1/6$, \eqref{sec-rough-inequ} implies that:
\begin{equation*}
\begin{split}
\left(\partial_t-\Delta_{g(t)}\right)&\left(\frac{|\Ob|^2_{g(t)}}{(1+\beta t\R_{g(t)})^{1/6}}\right)\leq\\
& \quad-\frac{1}{(1+\beta t\R_{g(t)})^{1/6}}|\nabla^{g(t)}\Ob|^2_{g(t)}+c(n)|\Rm(g(t))|_{g(t)}\frac{|\Ob|^2_{g(t)}}{(1+\beta t\R_{g(t)})^{1/6}}\\
&\quad-\frac{1}{12}\frac{|\nabla^{g(t)}(1+\beta t\R_{g(t)})|^2_{g(t)}}{(1+\beta t\R_{g(t)})^{1/6+2}}|\Ob|^2_{g(t)}-\frac{\beta}{6}\frac{\R_{g(t)}|\Ob|^2_{g(t)}}{(1+\beta t\R_{g(t)})^{1/6+1}}\\
&\leq  -\frac{1}{(1+\beta t\R_{g(t)})^{1/6}}|\nabla^{g(t)}\Ob|^2_{g(t)}-\frac{1}{12}\frac{|\nabla^{g(t)}(1+\beta t\R_{g(t)})|^2_{g(t)}}{(1+\beta t\R_{g(t)})^{1/6+2}}|\Ob|^2_{g(t)}\\
&\quad+\left(c(n)K(1+\beta t\R_{g(t)})-\frac{\beta}{6}\right)\R_{g(t)}\frac{|\Ob|^2_{g(t)}}{(1+\beta t\R_{g(t)})^{1/6+1}}\\
&\leq  -\frac{1}{(1+\beta t\R_{g(t)})^{1/6}}|\nabla^{g(t)}\Ob|^2_{g(t)}-\frac{1}{12}\frac{|\nabla^{g(t)}(1+\beta t\R_{g(t)})|^2_{g(t)}}{(1+\beta t\R_{g(t)})^{1/6+2}}|\Ob|^2_{g(t)}\\
&\quad+\left(c(n)K(1+\beta \varepsilon)-\frac{\beta}{6}\right)\R_{g(t)}\frac{|\Ob|^2_{g(t)}}{(1+\beta t\R_{g(t)})^{1/6+1}}.
\end{split}
\end{equation*}
Here we have used the assumption $|\Rm(g(t))|_{g(t)}\leq K\R_{g(t)}$ pointwise in the second line together with the bound $t\R_{g(t)}\leq \varepsilon = \frac{1}{c(n)K 12}$  in the last line. 

Then using  $c(n)K\varepsilon\leq\frac{1}{12}$ and $\beta (=24c(n)K),$ we have  $c(n)K(1+\beta \varepsilon)-\beta/6= -2c(n)K$  and   hence

\begin{equation*}
\begin{split}
\left(\partial_t-\Delta_{g(t)}\right)\left(\frac{|\Ob|^2_{g(t)}}{(1+\beta t\R_{g(t)})^{1/6}}\right)&\leq  -\frac{1}{3^{1/6}}|\nabla^{g(t)}\Ob|^2_{g(t)}-\frac{1}{12}\frac{|\nabla^{g(t)}(1+\beta t\R_{g(t)})|^2_{g(t)}}{(1+\beta t\R_{g(t)})^{1/6+2}}\\
&\quad-2c(n)K\R_{g(t)}\frac{|\Ob|^2_{g(t)}}{(1+\beta t\R_{g(t)})^{1/6+1}}.
\end{split}
\end{equation*}
Here we have used that $1+\beta t\R_{g(t)}\leq 3$ by the choice on $\beta$ with respect to $\varepsilon$.
This ends the proof of the lemma.
\end{proof}

Lemma \ref{magic-comput} is used to get the following preliminary result in order to derive an $L^2_{loc}$ bound on the tensor $\Ob$.

\begin{prop}\label{oh-my-gosh-L2}
Let $(M,g(t))_{t\in [0,T)}$ be a smooth, complete, connected solution satisfying
$\Rc(g(t)) \geq 0,$ $|\Rm(t)|\leq D_0 /t$ and 
 the assumptions of Lemma \ref{magic-comput} for  $U:=B_{g_0}(o,2)\setminus\overline{B_{g_0}(o,1)}.$ If $\eta$ is a Perelman-type cut-off function on $M\times(0,T)$ with respect to a point $p\in B_{g_0}(o,2)$ and radii $r_1$ and $r_2$ such that $B_{d_0}(p,r_2)\subset U$,
\begin{equation}\label{firstvn}
\begin{split}
\frac{d}{dt}\int_M\frac{|\Ob(t)|^2_{g(t)}}{(1+\beta t\R_{g(t)})^{1/6}}\eta\,d\mu_{g(t)}\leq
C_{*}\int_M|\Ob(t)|^2_{g(t)}\frac{|\nabla^{g(t)}\eta|_{g(t)}^2}{\eta}\,d\mu_{g(t)},
\end{split}
\end{equation}
for a universal constant $C_{*}\in (0,\infty)$, for all $t\leq \min\left( T, \frac{B(r_2^2 -r_1^2)}{D_0^2}\right),$ for a constant $B(n)\in (0,\infty).$ 
 If we replace condition \eqref{ass-1-magic} of Lemma \ref{magic-comput}  by
 $|\Rm(g(t)|\leq K$  on $U:=B_{g_0}(o,2)\setminus\overline{B_{g_0}(o,1)}$ for all $t\in [0,T)$, then 
 we get
 \begin{equation}\label{secondat}
\begin{split}
\frac{d}{dt} \Big( e^{-C(K,n)t}\int_M |\Ob(t)|^2_{g(t)}  \eta\,d\mu_{g(t)}\Big)\leq
e^{-C(K,n)t} C_{*}(K)  \int_M|\Ob(t)|^2_{g(t)}\frac{|\nabla^{g(t)}\eta|_{g(t)}^2}{\eta}\,d\mu_{g(t)},
\end{split}
\end{equation}
  after reducing $T$ so that $T\leq 1$.

\end{prop}
\begin{proof}
According to Lemma \ref{magic-comput} and integration by parts,
\begin{equation}
\begin{split}\label{prelim-IBP-L2-T}
\frac{d}{dt}\int_M\bigg(&\frac{|\Ob|^2_{g(t)}}{(1+\beta t\R_{g(t)})^{1/6}}\bigg)\eta\,d\mu_{g(t)}=\int_M\left(\partial_t-\Delta_{g(t)}\right)\left(\frac{|\Ob|^2_{g(t)}}{(1+\beta t\R_{g(t)})^{1/6}}\right)\eta\,d\mu_{g(t)}\\
&\quad+\int_M\left(\frac{|\Ob|^2_{g(t)}}{(1+\beta t\R_{g(t)})^{1/6}}\right)\left(\partial_t+\Delta_{g(t)}-\R_{g(t)}\right)\eta\,d\mu_{g(t)}\\
&\leq -\int_M\left(\frac{1}{3^{1/6}}|\nabla^{g(t)}\Ob|^2_{g(t)}+\frac{1}{12}\frac{|\nabla^{g(t)}(1+\beta t\R_{g(t)})|^2_{g(t)}}{(1+\beta t\R_{g(t)})^{1/6+2}}|\Ob|^2_{g(t)}\right)\eta\,d\mu_{g(t)}\\
&\quad+\int_M\left(\frac{|\Ob|^2_{g(t)}}{(1+\beta t\R_{g(t)})^{1/6}}\right)\left(\partial_t+\Delta_{g(t)}\right)\eta\,d\mu_{g(t)}\\
&\leq -\int_M\left(\frac{1}{3^{1/6}}|\nabla^{g(t)}\Ob|^2_{g(t)}+\frac{1}{12}\frac{|\nabla^{g(t)}(1+\beta t\R_{g(t)})|^2_{g(t)}}{(1+\beta t\R_{g(t)})^{1/6+2}}|\Ob|^2_{g(t)}\right)\eta\,d\mu_{g(t)}\\
&\quad+2\int_M\left(\frac{|\Ob|^2_{g(t)}}{(1+\beta t\R_{g(t)})^{1/6}}\right)\Delta_{g(t)}\eta\,d\mu_{g(t)},
\end{split}
\end{equation}
where we have used the fact that $\R_{g(t)}\geq 0$ in the second inequality together with the fact that $\eta$ is a Perelman-type cut-off function in the last line. Now, observe that for any positive $\varepsilon_i$, $i=1,2$,
\begin{equation}
\begin{split}\label{derive-prod-T}
&\left|\nabla^{g(t)}\left(\frac{|\Ob|^2_{g(t)}}{(1+\beta t\R_{g(t)})^{1/6}}\right)\right|_{g(t)}|\nabla^{g(t)}\eta|_{g(t)}\leq\\
&\left( \left(\frac{1}{(1+\beta t\R_{g(t)})^{1/6}}\right)|\nabla^{g(t)}|\Ob|^2_{g(t)}|_{g(t)}+\frac{|\nabla^{g(t)}(1+\beta t\R_{g(t)})|_{g(t)}}{(1+\beta t\R_{g(t)})^{1/6+1}}|\Ob|^2_{g(t)}\right)|\nabla^{g(t)}\eta|_{g(t)}\\
&\leq \varepsilon_1|\nabla^{g(t)}\Ob|^2_{g(t)}\eta+\varepsilon_2\frac{|\nabla^{g(t)}(1+\beta t\R_{g(t)})|^2_{g(t)}}{(1+\beta t\R_{g(t)})^{1/6+2}}|\Ob|^2_{g(t)}\eta+C(\varepsilon_1,\varepsilon_2)|\Ob|^2_{g(t)}\frac{|\nabla^{g(t)}\eta|^2_{g(t)}}{\eta},
\end{split}
\end{equation}
where we have used that $\R_{g(t)}\geq 0$.

A further integration by parts based on the  previous estimates \eqref{prelim-IBP-L2-T} and \eqref{derive-prod-T} leads to:
\begin{equation*}
\begin{split}
&\frac{d}{dt}\int_M\left(\frac{|\Ob|^2_{g(t)}}{(1+\beta t\R_{g(t)})^{1/6}}\right)\eta\,d\mu_{g(t)}\leq-\left(\frac{1}{3^{1/6}}-\varepsilon_1\right)\int_M|\nabla^{g(t)}\Ob|^2_{g(t)}\eta\,d\mu_{g(t)}\\
&-\left(\frac{1}{12}-\varepsilon_2\right)\int_M\frac{|\nabla^{g(t)}(1+\beta t\R_{g(t)})|^2_{g(t)}}{(1+\beta t\R_{g(t)})^{1/6+2}}|\Ob|^2_{g(t)}\eta\,d\mu_{g(t)}+C(\varepsilon_1,\varepsilon_2)\int_M|\Ob|^2_{g(t)}\frac{|\nabla^{g(t)}\eta|^2_{g(t)}}{\eta}\,d\mu_{g(t)}\\
&\leq C(\varepsilon_1,\varepsilon_2)\int_M|\Ob|^2_{g(t)}\frac{|\nabla^{g(t)}\eta|^2_{g(t)}}{\eta}\,d\mu_{g(t)},
\end{split}
\end{equation*}
provided $\varepsilon_1$ and $\varepsilon_2$ are chosen small enough and where $C(\varepsilon_1,\varepsilon_2)\in (0,\infty)$ is a constant that may differ from that of estimate \eqref{derive-prod-T}.

 If we replace condition \eqref{ass-1-magic} of Lemma \ref{magic-comput}  by
 $|\Rm(g(t)|\leq K$ on $U$ for all $t\in [0,T)$, then \eqref{secondat} holds.  This can be seen by considering 
 the equation \eqref{replaceeq} in place of 
 \eqref{est-T-heat-subsol} in the above proof.
\end{proof}

We are now in a position to establish a faster than polynomial decay for the $L^2_{loc}$ norm of the tensor $\Ob(t)$ as $t$ goes to $0$:
\begin{prop}\label{theo-dream-T}
Let $(M,g(t))_{t\in [0,T)}$ be a smooth, complete, connected solution satisfying
$\Rc(g(t)) \geq 0,$ $|\Rm(t)|\leq D_0 /t$ and 
 the assumptions of Lemma \ref{magic-comput} for  $U:=B_{g_0}(o,2)\setminus\overline{B_{g_0}(o,1)}.$
Let  $p\in B_{g_0}(o,2)$ and  $r>0$ a radius such that $B_{g_0}(p,2r)\subset U$.
 If  $$\int_{B_{g_0}(o,3)} |\Ob(0)|^2_{g_0}\,d\mu_{g_0} \leq \varepsilon \leq 1,$$ then for each $k\geq 0$, there exist   $C_k=C_k(n,r,A,K,D_0)\in (0,\infty)$ such that,
 \begin{equation*}  
\int_{B_{g_0}\left(p,2^{-k}r\right)} |\Ob(t)|^2_{g(t)}  \,d\mu_{g(t)}\leq \int_{B_{g(t)}\left(p,2^{-k}r\right)} |\Ob(t)|^2_{g(t)}  \,d\mu_{g(t)}
\leq C_k t^k +  C_k\varepsilon,
\end{equation*}
 for $t\in[0,\min\{T,B(n)2^{-2k}r^2/D_0^2\}]$.
 
If we replace condition \eqref{ass-1-magic} of Lemma \ref{magic-comput}  by
 $|\Rm(g(t)|\leq K$  on $U:=B_{g_0}(o,2)\setminus\overline{B_{g_0}(o,1)}$ for all $t\in [0,T)$, then
 \begin{equation*}  
\int_{B_{g_0}\left(p,2^{-k}r\right)} |\Ob(t)|^2_{g(t)}  \,d\mu_{g(t)}\leq C_k \int_{B_{g(t)}\left(p,2^{-k}r\right)} |\Ob(t)|^2_{g(t)}  \,d\mu_{g(t)}
\leq C^2_k t^k +  C^2_k\varepsilon.
\end{equation*}

\end{prop}

\begin{proof}
Let $\eta$ be a Perelman type cut-off function with respect to $p$ and radii $r_1:=\alpha r$ and $r_2:=\beta r$ with $0<\alpha<\beta\leq 1$.
 Using Proposition \ref{oh-my-gosh-L2}, we see that
\begin{equation*}
 \frac{d}{dt}\int_M\frac{|\Ob|^2_{g(t)}}{(1+\beta t\R_{g(t)})^{1/6}}\eta\,d\mu_{g(t)}\leq C(n,\alpha,\beta,r)\int_M|\Ob|^2_{g(t)}\frac{|\nabla^{g(t)}\eta|_{g(t)}^2}{\eta}\,d\mu_{g(t)}.
\end{equation*}
 Integrating in time implies,
 \begin{equation}\label{first-step-T-L2}
 \begin{split}
 &\int_{B_{g_0}(p,\alpha r)} |\Ob(t)|^2_{g(t)}\,d\mu_{g(t)}\leq  \int_{B_{g(t)}(p,\alpha r)} |\Ob(t)|^2_{g(t)}\,d\mu_{g(t)}\\
 &\leq C(n,\al,\beta,r) \int_{B_{g_0}(p,\beta r)} |\Ob(0)|^2\,d\mu_{g(0)} + C(n,\alpha,\beta,r)\int_0^t\int_{B_{g(s)}(p,\beta r)}|\Ob(s)|^2_{g(s)}\,d\mu_{g(s)}ds.
  \end{split}
 \end{equation}
 In particular, if $\alpha:=1/2$ and $\beta:=1$ then
 since $B_{g_0}(p,r/2)\subset B_{g(t)}(p,r/2)$ and $B_{g(s)}(p,r)\subset B_{g_0}(p,r+C\sqrt{s})\subset B_{g_0}(p,2r)\subset B_{g_0}(o,2)$ according to Proposition \ref{prop:RF_basics} and by assumption on $t\in[0,\min\{T,r^2/C^2\}]$:
  \begin{equation*}
 \begin{split}
 \int_{B_{g_0}(p,r/2)}& |\Ob(t)|^2_{g(t)}\,d\mu_{g(t)}\leq  \int_{B_{g(t)}(p,r/2)} |\Ob(t)|^2_{g(t)}\,d\mu_{g(t)}\\
 & \leq C(n,r) \int_{B_{g_0}(o,3)} |\Ob(0)|^2\,d\mu_{g(0)} + C(n,r)\int_0^t\int_{B_{g_0}(p,r+C\sqrt{s})}|\Ob(s)|^2_{g(s)}\,d\mu_{g(s)}ds\\
  & \leq C(n,r) \int_{B_{g_0}(o,3)} |\Ob(0)|^2\,d\mu_{g(0)} + C(n,r)\int_0^t\int_{B_{g_0}(p,2r)}|\Ob(s)|^2_{g(s)}\,d\mu_{g(s)}ds\\
 &\leq C_0(1 + \varepsilon), 
 \end{split}
 \end{equation*}
 for some $C_0=C_0(n,r,A,K,D_0)\in (0,\infty)$. Here we have used Lemma \ref{coro-rough-T} in the last line.
 In particular, we have obtained:
  \begin{equation}\label{first-step-T-L2-bis}
  \int_{B_{g(t)}(p,r/2)} |\Ob(t)|^2_{g(t)}\,d\mu_{g(t)}\leq C_0(1 + \varepsilon).
 \end{equation}
Considering  the estimate  \eqref{first-step-T-L2}  for 
 $\alpha=1/4$ 
and $\beta=1/2$  and using the already obtained   bound  \eqref{first-step-T-L2-bis}  in this estimate yields
\begin{equation*}
 \int_{B_{g_0}(p,r/4)} |\Ob(t)|_{g(t)}^2\,d\mu_{g(t)}   \leq  \int_{B_{g(t)}(p,r/4)} |\Ob(t)|^2_{g(t)}\,d\mu_{g(t)}\leq    C_0 \varepsilon+ C_0(1 + \varepsilon)t \leq C_1\varepsilon+  C_1t,
 \end{equation*}
 where we have used the fact that $\varepsilon\leq 1$. By induction on $k$, one ends up with the expected decay on smaller and smaller balls of radii $r/2^k$.
 
 If we replace condition \eqref{ass-1-magic} of Lemma \ref{magic-comput}  by
 $|\Rm(g(t)|\leq K$ on $U:=B_{g_0}(o,2)\setminus\overline{B_{g_0}(o,1)}$ for all $t\in T$, then \eqref{secondat}  can be used  in place of 
 \eqref{firstvn} in the above proof, and hence the results still hold.
  \end{proof}

The next lemma will be used twice. On the one hand, we will get a rough lower bound on the mean value of the function $v$ before passing to the limit: this is the content of Corollary \ref{mean-value-coro-I}. On the other hand, this lemma will be invoked once more in Corollary \ref{coro-mean-value-II}
to get a faster than polynomial lower bound on the mean value of $v$ once we take a limit along a sequence of times going to $0$. 
\begin{lemma}\label{lemma-mean-value-v-minus}
Under the assumptions of Proposition \ref{prop-grad-bd-a-priori},    let $p\in B_{g_0}(o,2)$ and  $r>0$  a radius such that $B_{g_0}(p,2r)\subset B_{g_0}(o,2)\setminus\overline{B_{g_0}(o,1)}$.

If $\int_{B_{g_0}(o,3)} v_+^2(0)\,d\mu_{g_0}\leq\varepsilon \leq 1$ then if $\eta$ is a Perelman type cut-off function with respect to $p$ and radii $r_1:=r/2$ and $r_2:=r$,  for $t\in[0,\min\{T,B(n)r^2/D_0^2\}]$, where $B(n)\in (0,\infty)$ is a uniform constant,
\begin{equation*}
\begin{split}
\int_{B_{g_0}(p,r)}v\eta\,d\mu_{g(t)}&\geq \int_Mv\eta\,d\mu_{g(0)}-2\int_{0}^t\int_{B_{g_0}(p,2r)}|\Ob(s)|^2_{g(s)}\,d\mu_{g(s)}ds\\
&\quad-C(n)\int_{0}^t\int_{B_{g_0}(p,2r)}|s\div_{g(s)}\Ob(s)+\Ob(s)(\nabla^{g(s)}u(s))|_{g(s)}\,d\mu_{g(s)}ds\\
&\quad-\int_{0}^t\int_{B_{g_0}(p,2r)}\left(\sup_{B_{g_0}(p,2r)\times [0,s]}v_+\right)\R_{g(s)}\,d\mu_{g(s)}ds-C't^{1/2}\sup_{B_{g_0}(p,2r)\times [0,t]}v_+,
\end{split}
\end{equation*}
where $C'=C'(r,n,A,D_0)\in (0,\infty)$.

\end{lemma}

\begin{proof}
Let $\eta$ be such a Perelman type cut-off function and let us multiply the evolution equation satisfied by $v$ from Proposition \ref{prop-evo-eqn-v} by $\eta$ so that after integration we get
\begin{equation*}
\begin{split}
\int_Mv\eta\,d\mu_{g(s)}\bigg\vert_{0}^t&=\int_{0}^t\int_M-2|\Ob(s)|^2_{g(s)}\eta+v\left(\partial_s+\Delta_{g(s)}-\R_{g(s)}\right)\eta\,d\mu_{g(s)}ds\\
&\geq-2\int_{0}^t\int_{B_{g_0}(p,r+C\sqrt{s})}|\Ob(s)|^2_{g(s)}\,d\mu_{g(s)}ds\\
&\quad+\int_{0}^t\int_M\underbrace{\left(v-\sup_{B_{g_0}(p,2r)\times [0,s]}v_+\right)}_{\leq 0}\underbrace{\left(\partial_s+\Delta_{g(s)}-\R_{g(s)}\right)\eta}_{\leq 2\Delta_{g(s)}\eta}\,d\mu_{g(s)}ds\\
&\quad+\int_{0}^t\sup_{B_{g_0}(p,2r)\times [0,s]}v_+\int_M\left(\partial_s+\Delta_{g(s)}-\R_{g(s)}\right)\eta\,d\mu_{g(s)}ds\\
&\geq -2\int_{0}^t\int_{B_{g_0}(p,2r)}|\Ob(s)|^2_{g(s)}\,d\mu_{g(s)}ds-2\int_{0}^t\int_M|\nabla^{g(s)}v|_{g(s)}|\nabla^{g(s)}\eta|_{g(s)}\,d\mu_{g(s)}ds\\
&\quad-\int_{0}^t\sup_{B_{g_0}(p,2r)\times [0,s]}v_+\int_{B_{g_0}(p,2r)}\R_{g(s)}\,d\mu_{g(s)}ds-C't^{1/2}\sup_{B_{g_0}(p,2r)\times [0,t]}v_+,
\end{split}
\end{equation*}
where $C'=C'(r,n,A,D_0)\in (0,\infty)$.
Here we have used that $v$ is  bounded by Definition \ref{defn-v} and the interior estimates from Proposition \ref{prop-grad-bd-a-priori} together with the fact that $\partial_s\eta\geq -c(r)/s^{1/2}$ for a   constant $c(r)$ depending on $r$ in the last inequality thanks to Lemma \ref{lemma-Cutoff-Perelman}.
According to [\eqref{beautiful-covid19-bis}, Lemma \ref{lemma-gen-form-exp-ric-pin}
], we infer that
\begin{equation*}
\begin{split}
\int_Mv\eta\,d\mu_{g(t)}&\geq \int_Mv\eta\,d\mu_{g(0)}-2\int_{0}^t\int_{B_{g_0}(p,2r)}|\Ob(s)|^2_{g(s)}\,d\mu_{g(s)}ds\\
&\quad-C(n)\int_{0}^t\int_{B_{g_0}(p,2r)}|s\div_{g(s)}\Ob(s)+\Ob(s)(\nabla^{g(s)}u(s))|_{g(s)}\,d\mu_{g(s)}ds\\
&\quad-\int_{0}^t\int_{B_{g_0}(p,2r)}\left(\sup_{B_{g_0}(p,2r)\times [0,s]}v_+\right)\R_{g(s)}\,d\mu_{g(s)}ds-C't^{1/2}\sup_{B_{g_0}(p,2r)\times [0,t]}v_+,
\end{split}
\end{equation*}
which turns out to be the exact inequality we are looking for.
\end{proof}

As explained before the statement of Lemma \ref{lemma-mean-value-v-minus}, the following corollary is a first step to show that the mean value of the function $v(t)$ converges to $0$ as $t$ goes to $0$.

\begin{coro}\label{mean-value-coro-I}
Under the assumptions of Lemma \ref{magic-comput} with $U:=B_{g_0}(o,2)\setminus\overline{B_{g_0}(o,1)}$, and Proposition \ref{prop-grad-bd-a-priori},   let $p\in B_{g_0}(o,2)$ and  $r>0$  a radius such that $B_{g_0}(p,2r)\subset U$.\\

If $\int_{B_{g_0}(o,3)} \left(v^2(0)+|\Ob(0)|^2_{g_0}\right)\,d\mu_{g_0}\leq\varepsilon \leq 1$ then for $t\in[0,\min\{T,B(n)r^2/D_0^2\}]$, where $B(n)\in (0,\infty)$ is a uniform constant,
\begin{equation*}
\begin{split}
\int_{B_{g_0}(p,r/2)}v(t)\,d\mu_{g(t)}&\geq -C_0\,\varepsilon^{1/2}-\varepsilon(t),
\end{split}
\end{equation*} 
where $C_0=C_0(n,r,A,K,D_0)\in (0,\infty)$ and
\begin{equation*}
\ep(t) := C(n,A,K,D_0) (\vol_{g(t)}B_{g_0}(p,2r)-\vol_{g_0}B_{g_0}(p,2r)) +C(n,A,K,D_0,r)t^{  1 /2 }  \to 0,
\end{equation*}
as $t\to 0$.

If we replace condition (1) of Lemma \ref{magic-comput}  by
 $|\Rm(g(t)|\leq K$  on $U:=B_{g_0}(o,2)\setminus\overline{B_{g_0}(o,1)}$ for all $t\in T$, then the result   still  holds. 

\end{coro}

\begin{proof}
According to Proposition \ref{theo-dream-T} with $k=0$ and Lemma \ref{lemma-mean-value-v-minus} applied to $r_1:=r/2$ and $r_2:=r$ so that $\eta(\cdot,t)$  is constant in space on $B_{g_0}(o,r/2)\subset B_{g(t)}(o,r/2)$,  and Lemma \ref{lemma-L2-v-plus}, we get
 \begin{eqnarray*}
&&\int_{B_{g_0}(p,r/2)}v(t)\,d\mu_{g(t)} \\
&& = e^{c(r)t} \int_{B_{g_0}(p,r/2)}\eta(t)v(t)\,d\mu_{g(t)} \\
&& =  e^{c(r)t} \int_{M}\eta(t)v(t)\,d\mu_{g(t)} -e^{c(r)t}\int_{  B_{g_0}(p,r ) \backslash  B_{g_0}(p,r/2)}\eta(t)v(t)\,d\mu_{g(t)}
 \\
&&\geq -C_0(n,r,A,K,D_0)\varepsilon^{1/2}-C_0(n,r,A,K,D_0)(t+\varepsilon)\\
&& \ \ \ \ \ \ -C(n,r,A,K,D_0)\int_{0}^t\int_{B_{g_0}(p,r)}\R_{g(s)}\,d\mu_{g(s)}ds-C(n,r,A,K,D_0)t^{1/2},
\\
&& \ \ \ \ \ \ -C(n)\int_{0}^t\int_{B_{g_0}(p,2r)}|s\div_{g(s)}\Ob(s)+\Ob(s)(\nabla^{g(s)}u(s))|_{g(s)}\,d\mu_{g(s)}ds\\
\end{eqnarray*}
where we have used that $v_+$ is  bounded thanks to  the interior estimate from Proposition \ref{prop-grad-bd-a-priori}. Now, thanks to the interior estimate on $\nabla^{g(t)}\Ob(t)$ from [\eqref{loc-upper-bd-grad-ob}, Proposition \ref{prop-grad-bd-a-priori}], one ends up with
 \begin{equation*}
\begin{split}
\int_{B_{g_0}(p,r/2)}v(t)\,d\mu_{g(t)}&\geq -C_0(n,r)\varepsilon^{1/2}-C(n,A,K,D_0,r)t^{1/2}\\
&\quad-C(n,A,K,D_0)\int_{0}^t\int_{B_{g_0}(p,r)}\R_{g(s)}\,d\mu_{g(s)}ds.
\end{split}
\end{equation*}

Now observe that:
\begin{equation*}
-\int_{0}^t\int_{B_{g_0}(p,2r)}\R_{g(s)}\,d\mu_{g(s)}ds=\vol_{g(t)}B_{g_0}(p,2r)-\vol_{g_0}B_{g_0}(p,2r).
\end{equation*}
Therefore, if 
 \begin{equation*}
  \ep(t):= -C(n,A,K,D_0)(\vol_{g(t)}B_{g_0}(p,2r)-\vol_{g_0}B_{g_0}(p,2r)) +C(n,A,K,D_0,r)t^{1/2 },
  \end{equation*}
   then $\ep(t)  \to 0$ as $t\downto 0, $  by smoothness.   
\end{proof}
 
\section{Pointwise decay on the limit solution}\label{sec-pt-dec-sol}

\noindent In this section, we   use the notation from section \ref{sec-equations}. 
We consider $u$ as constructed in \eqref{eq:def-limit-solution} solving \eqref{eq:eqn-limit-solution}, i.e.~$u$ is the local smooth limit as $i\to \infty$ of solutions $u^i$ to \eqref{eqn-heat-sing-bis-rep} along the translated Ricci flows $(0,T)\ni t \mapsto g_i(t).$ We denote by $\Ob^i(t)$, respectively $v^i(t)$, the corresponding tensor $\Ob(t)$, respectively $v(t)$, associated to the solution $u^i(t)$ for each $i$. We will however stick to $\Ob(t)$ and $v(t)$ as long as the solution $u(t)$ is concerned. 

Moreover, we assume that the assumptions of Theorem \ref{main-I} hold. In particular, the assumptions of Lemma \ref{magic-comput} with $U:=B_{d_0}(o,2)\setminus\overline{B_{d_0}(o,1)}$ are satisfied thanks to \cite[Lemma $A.2$]{DSS-II}: we fix $p\in B_{d_0}(o,2)$ and a radius $r>0$ such that $B_{d_0}(p,2r)\subset U$ once and for all.
The first result, explains how to show that the limit as $i\to \infty$ of the sequences $(u^i)_i$ and $(g_i)_i$ exist, and proves  estimates on how 
the initial value $u_0$ of $u$ is attained. Here, the constant $A$ of Sections \ref{section-pt-int-est} and \ref{sec-int-est}  will correspond to a constant $A = C(n,T)$ in view of Proposition \ref{ptwise-bd-ui}. 
\begin{rk} \label{replace2}
In the case that the cone in the conditions of Theorem 
\ref{main-I} is smooth away from the tip, 
 \cite{DSS-1} ensures that  
\begin{eqnarray}
 |\Rm(g(t))|\leq K \ \  \text{ on } \ \ U:=B_{d_0}(o,2)\setminus\overline{B_{d_0}(o,1)} 
 \end{eqnarray}
  for all $t\in (0,\hat T)$ for some constant $K,$ for some $0<\hat T\leq T.$  
  See \cite{DSS-1} for more details.
  \end{rk}

\begin{prop}\label{coro-dist-comp-sol}
Let $(M,g(t))_{t\in (0,T)}$ be a solution satisfying the assumptions in Theorem \ref{main-I}, and let $u^i$, $g_i$  be the solutions constructed in Section   \ref{section-can-fct}. 
On $B_{d_0}(o,3)$, there exists $C\in (0,\infty)$ such that for $t\in(0,T)$:
\begin{equation*}
\left|u^i(t)-\frac{d_{g_i(0)}(o,\cdot)^2}{4}\right|\leq C\sqrt{t}+\varepsilon_i, 
\end{equation*}
In particular, on $B_{d_0}(o,3)\times (0,T)$ 
\begin{equation}\label{eq:u-lower-bound}  
\left|u^i(t)-\frac{d_{g_i(t)}(o,\cdot)^2}{4}\right|\leq C\sqrt{t}+\varepsilon_i, 
\end{equation}
Hence taking a limit in $i$ and using the interior estimates of section \ref{section-pt-int-est} we obtain limits  $u = \lim_{i\to \infty} u^i,$  $g = \lim_{i\to \infty} g_i ,$ $\Ob = \lim_{i\to \infty} \Ob^i$  and  $v = \lim_{i\to \infty} v^i,$ where the limits are locally, smoothly defined  for $t>0,$ and $u(\cdot,t) \to d_0^2(o,\cdot)/4,$ locally uniformly as $t\downto 0.$ 
\end{prop}

\begin{proof}
Thanks to Propositions \ref{prop-ell-reg}, \ref{ptwise-bd-ui}, Proposition \ref{prop-grad-bd-a-priori} is applicable 
with $A = C(n,T)$ and we get  via the heat equation satisfied by $u^i$:
\begin{equation*}
\left|\partial_tu^i\right|\leq \frac{C_2}{\sqrt{t}},\quad t\in(0,T).
\end{equation*}
In particular, this implies after integration between $0$ and $t$ on $B_{d_0}(o,3)$:
\begin{equation*}
\left|u^i(t)-u^i(0)\right|\leq C\sqrt{t},\quad t\in(0,T).
\end{equation*}
The triangular inequality then gives the expected estimate thanks to [\eqref{linfty-u-init}, Proposition \ref{prop-ell-reg}] on $B_{d_0}(o,3)$: 
\begin{equation*}
\left|u^i(t)-\frac{d_{g_i(0)}(o,\cdot)^2}{4}\right|\leq C\sqrt{t}+\varepsilon_i,\quad t\in(0,T).
\end{equation*}
Combining this estimate with \eqref{eq:distance-distortion} leads to \eqref{eq:u-lower-bound}.
\end{proof}
The next result of this section guarantees that the tensors $\Ob^i(t)$ are decaying pointwise to $0$ as $t$ goes to $0$ with a rate faster than polynomial.

\begin{coro}\label{coro-pointwise-T}
Let $(M,g(t))_{t\in (0,T)}$ be a solution satisfying the assumptions in Theorem \ref{main-I}, and let $u^i$, $g_i$  be the solutions constructed in Section   \ref{section-can-fct}.Then there exists $0<S\leq T$ such that, if  $$\int_{B_{g_i(0)}(o,3)} |\Ob^i(0)|^2_{g_i(0)}\,d\mu_{g_i(0)}\leq\varepsilon_i,$$ then for each $k\geq 0$, there exist $C\in (0,\infty)$ and $C_k\in (0,\infty)$ such that for $t\in[0,S)$ satisfying $B_{g_i(0)}(p,\sqrt{t})\times (2^{-1}t,2t)\subset B_{g_i(0)}(p, 2^{-k}r)\times (0,\min\{S,2^{-2k}r^2/C^2\})$ then for $i\geq 0$:
\begin{equation*}
|\Ob^i(p,t)|_{g_i(t)}^2\leq C_kt^{-n/2}\left(t^{k}+\varepsilon_i\right).
\end{equation*}
In particular, for $k\geq 0$ and $l\geq 0$, there exist $C_{k,l}\in (0,\infty)$ and $T_k\in (0,\infty)$ such that on $B_{d_0}(o, 2)\setminus\overline{B_{d_0}(o,1)}\times(0,T_k)$,
\begin{equation*}
|\nabla^{g(t),\,l}\Ob(t)|_{g(t)}^2\leq C_{k,l}\,t^{k}.
\end{equation*}
\end{coro}

\begin{proof}
We recall that $U:=B_{d_0}(o,2)\setminus\overline{B_{d_0}(o,1)}$ satisfies condition $\eqref{ass-2-magic}$ of Lemma \ref{magic-comput} thanks to \cite[Lemma A.2]{DSS-II}, after reducing $T>0$ to $S>0$ once if necessary.
Condition \eqref{ass-1-magic} of Lemma \ref{magic-comput} is satisfied by assumption or, alternatively 
$|\Rm(g(t))|\leq K$  on $U:=B_{d_0}(o,2)\setminus\overline{B_{d_0}(o,1)}$ is satisfied in the case of a smooth cone away from the tip, in view of 
Remark \ref{replace2}.
Also, from Proposition \ref{prop-eqn-evo-obs-tensor}, we see  that each $\Ob^i(t)$ satisfies schematically $(\partial_t-\Delta_{g_i(t)})\Ob^i(t)= \Rm(g_i(t)) \ast_{g_i(t)}\Ob^i(t)$. In particular, since $|\Rm(g_i(t))|_{g_i(t)}\leq D_0/t$, $|\Ob^i(t)|^2_{g_i(t)}$ is a subsolution to the following heat equation along the Ricci flow $(g_i(t))_t$:
\begin{equation*} 
\left(\partial_t-\Delta_{g_i(t)}\right)|\Ob^i(t)|^2_{g_i(t)}\leq -2|\nabla^{g_i(t)}\Ob^i(t)|^2_{g_i(t)}+\frac{c(n)D_0}{t}|\Ob^i(t)|^2_{g_i(t)},
\end{equation*}
so that $t^{-c(n)D_0}|\Ob^i(t)|^2_{g(t)}$ is a subsolution to the heat equation.

 Choose  $k>c(n)D_0+n$ and perform a local Nash-Moser iteration  on each ball $B_{g_i(0)}(p,\sqrt{t})\times\left(t,2t\right)\subset B_{g_i(0)}(p,2^{-k}r)\times(0,\min\{S,2^{-2k}r^2/C^2\})$ to get for each $\theta\in(0,1)$,  
\begin{equation}\label{loc-nash-moser-T}
\sup_{B_{g_i(0)}(p,\sqrt{\theta t})\times\left(t(1+\theta),2t\right)}\!\!\!\!\!\!\!\!\!\!\!\!\!\!\!s^{-c(n)D_0}|\Ob^i(s)|^2_{g_i(s)}\leq C(n,\theta,D_0)\, \dashint_t^{2t}\dashint_{B_{g_i(0)}(p,\sqrt{t})}\!\!\!\!\!\!\!\!\!\!\!\!\!\!\!\!\!\!\!s^{-c(n)D_0}|\Ob^i(s)|^2_{g_i(s)}\,d\mu_{g_i(s)}ds.
\end{equation} 
 Again, see for instance 
   \cite{Saloff-Coste-Riem}, \cite{Grig92} or \cite[Theorem 25.2]{Chow2010TheRF}  
   for a proof.
 
Now  apply Proposition  \ref{theo-dream-T} so that the previous inequality \eqref{loc-nash-moser-T} leads to the pointwise bound:  
\begin{equation*}
\begin{split}
\sup_{B_{g_i(0)}(p,\sqrt{\theta t})\times\left(t(1+\theta),2t\right)}&s^{-c(n)C/2}|\Ob^i(s)|^2_{g_i(s)} \\
&\leq C(n,k,\theta,D_0)\big(\tfrac{r^2}{t}\big)^{\frac{n}{2}}\dashint_t^{2t}s^{-c(n)C/2}\dashint_{B_{g_i(0)}(p,2^{-k}r)}|\Ob^i(s)|^2_{g_i(s)}\,d\mu_{g_i(s)}ds\\
&\leq C(n,k,\theta,D_0)t^{-\frac{n}{2} -c(n)C/2}\dashint_t^{2t}s^k+\varepsilon_i\,ds\\
&\leq  C(n,k,\theta,D_0) t^{-n/2-c(n)C/2} \Big( t^k+\varepsilon_i\Big),\\
\end{split}
\end{equation*}
for $t\in(0,\min\{S,2^{-2k}r^2/C^2\})$ as expected.

The corresponding decay in time for $\Ob(t)$ associated to the pointwise limit $u(t)$ follows by letting $\varepsilon_i$ going to $0$ in the previous estimate: this gives us the desired decay on each ball $B_{d_0}(p,2^{-k}r)\subset B_{d_0}(o, 2)\setminus\overline{B_{d_0}(o,1)}$. If $k\geq 0$ is given, $\Ob(t)$ will decay similarly on $B_{d_0}(o, 2)\setminus\overline{B_{d_0}(o,1)}$ by applying the previous bound to a covering of this set by balls of the form $B_{d_0}(p,2^{-k}r)$. Higher covariant derivatives follow by interpolation with interior bounds on covariant derivatives of $\Ob(t)$ given by Proposition \ref{prop-grad-bd-a-priori}.
\end{proof}

We are then in a position to prove a lower bound on the mean value of the function $v(t)$:

\begin{coro}\label{coro-mean-value-II}
Let $(M,g(t))_{t\in (0,T)}$ be a solution satisfying the assumptions in Theorem \ref{main-I},   $v$ the limit of the 
$v^i$'s, from Theorem \ref{coro-dist-comp-sol}.
Let $p\in B_{d_0}(o,2)$ and a radius $r>0$ such that $B_{d_0}(p,2r)\subset B_{d_0}(o,2)\setminus \overline{B_{d_0}(o,1)}$. Then, there exists $0<S \leq T$ , such that for each $k> n/2$ and $r>0$, there exist  $C_k\in (0,\infty)$ and $C\in (0,\infty)$ such that for $t\in(0,\min\{S,2^{-2k}r^2/C^2\})$:
 \begin{equation*}
\begin{split}
\int_{B_{d_0}(p,2^{-k}r)}v(t)\,d\mu_{g(t)}\geq -C_kt^{k/2-n/4}.
\end{split}
\end{equation*}
\end{coro}

\begin{proof}
We recall that $U:=B_{d_0}(o,2)\setminus\overline{B_{d_0}(o,1)}$ satisfies condition $\eqref{ass-2-magic}$ of Lemma \ref{magic-comput} thanks to \cite[Lemma A.2]{DSS-II}, after reducing $T>0$ to $S>0$ once if necessary.
Condition \eqref{ass-1-magic} of Lemma \ref{magic-comput} is satisfied by assumption or, alternatively 
$|\Rm(g(t))|\leq K$  on $U:=B_{d_0}(o,2)\setminus\overline{B_{d_0}(o,1)}$ is satisfied in the case of a smooth cone away from the tip, in view of 
Remark \ref{replace2}.
For $k> n/2$, let us apply Proposition \ref{theo-dream-T}  for  $\varepsilon=\varepsilon_i$ for each $i$ together with Lemma \ref{lemma-mean-value-v-minus} applied to $r_1:=2^{-k-1}r$ and $r_2:=2^{-k}r$ so that each cut-off function $\eta^i(\cdot,t)$ is constant on $B_{g_i(0)}(p,2^{-k-1}r)\subset B_{g_i(t)}(p,2^{-k-1}r)$ and let us pass to the limit as $\varepsilon_i$ converges to $0$ to get for $0<t'<t$,
 \begin{equation*}
\begin{split}
\int_{B_{d_0}(p,2^{-k-1}r)}&v\,d\mu_{g(t)}\geq \\
&\int_{M}\eta v\,d\mu_{g(t')} -C_kt^{k+1}-\int_{0}^t\int_{B_{d_0}(o,2^{1-k}r)}\left(\sup_{B_{d_0}(p,2^{1-k}r)\times [0,s]}v_+\right)\R_{g(s)}\,d\mu_{g(s)}ds\\
&\quad-C_kt^{k/2+1}-Ct^{1/2}\left(\sup_{B_{d_0}(p,2^{1-k}r)\times [0,t]}v_+\right),
\end{split}
\end{equation*}
where we have used interior estimates on $\nabla^{g(t)}\Ob(t)$ from [\eqref{loc-upper-bd-grad-ob}, Proposition \ref{prop-grad-bd-a-priori}] together with Corollary \ref{coro-pointwise-T}. Now, Corollary \ref{coro-sup-v_+} applied to $\varepsilon=\varepsilon_i$ ensures once $\varepsilon_i$ is sent to $0$ that:
 \begin{equation*}
\begin{split} 
\int_{B_{d_0}(p,2^{-k-1}r)}v\,d\mu_{g(t)}&\geq \int_{M}\eta v\,d\mu_{g(t')} -C_kt^{k+1}-C_k\int_{0}^t\int_{B_{d_0}(p,2^{1-k}r)}s^{k/2-n/4}\R_{g(s)}\,d\mu_{g(s)}ds\\
&\quad-C_kt^{k/2+1}-C_kt^{1/2+k/2-n/4}\\
&\geq \int_{M}\eta v\,d\mu_{g(t')} -C_kt^{k/2+1/2-n/4}-C_k\int_{0}^ts^{k/2-n/4-1}\,ds\\
&\geq \int_{M}\eta v\,d\mu_{g(t')} -C_kt^{k/2-n/4},
\end{split}
\end{equation*}
for some positive constant $C_k$ that may vary from one estimate to another. Here we have used the rough bound $\R_{g(s)}\leq C/s$ in the penultimate line. Invoking Corollary \ref{mean-value-coro-I}, using Cheeger-Colding's volume convergence Theorem,  and letting $t'$ going to $0$,  leads to the desired result.
\end{proof}

As a consequence of the previous Corollary \ref{coro-mean-value-II}, we get a faster than polynomial decay on $v$ together with the expected gradient bound from [\eqref{grad-bd-main-I}, Theorem \ref{main-I}]:
\begin{coro}\label{grad-low-bd-v}
Let $(M,g(t))_{t\in (0,T)}$ be a solution satisfying the assumptions in Theorem \ref{main-I}, $v$ the limit of the 
$v^i$'s, from Theorem \ref{coro-dist-comp-sol}. 
For $k\geq 0$, there exists $C>0$ and $T_k>0$ such that on $B_{d_0}(o, 2)\setminus\overline{B_{d_0}(o,1)}\times(0,T_k)$:
 \begin{equation*}
\begin{split}
|v(t)|\leq C_kt^{k}.
\end{split}
\end{equation*}
In particular, 
\begin{equation*}
\begin{split}
||\nabla^{g(t)}u(t)|^2_{g(t)}-u(t)|\leq Ct.
\end{split}
\end{equation*}
\end{coro}

\begin{proof}
According to Corollary \ref{coro-pointwise-T}, the tensor $\Ob(t)$ decays faster than any polynomial in time. By interior estimates, so are its covariant derivatives, a fact that lets us conclude that the gradient of the function $v$ does decay as fast as any polynomial as $t$ goes to $0$ thanks to [\eqref{beautiful-covid19-bis}, Lemma \ref{lemma-gen-form-exp-ric-pin}]. Now the combination of Lemma \ref{lemma-L2-v-plus} with $\varepsilon=0$ and Corollary \ref{coro-mean-value-II} shows that the mean value of $v$ on balls $B_{d_0}(p,2^{-k}r)$ decays as fast as $t^{k/2-n/4}$. The mean value theorem then implies the expected pointwise behavior on $v$.

By recalling the definition of $v$, Definition \ref{defn-v}, the previously established rate on $v(t)$ as $t$ goes to $0$ shows that the difference $|\nabla^{g(t)}u(t)|^2_{g(t)}-u(t)$ is measured by $t^2\R_{g(t)}$ up to an error that decays faster than any polynomial. By assumption $t|\Rm(g(t))|_{g(t)} \leq D_0$  on the flow, and hence, this implies the result. 
\end{proof}

We can finally prove exponential decay of the tensor $\Ob(t)$ together with its covariant derivatives as $t$ goes to $0$:
\begin{theo}\label{theo-exp-dec}
Let $(M,g(t))_{t\in (0,T)}$ be a solution satisfying the assumptions in Theorem \ref{main-I},  $\Ob$ the limit of the 
$\Ob^i$'s, from Theorem \ref{coro-dist-comp-sol}. 
For $k\geq 0$, there exists $C\in (0,\infty)  $ and $C_k\in (0,\infty)$ such that on $B_{d_0}(o,2)\setminus \overline{B_{d_0}(o,1)}\times(0,\min\{T,C^{-2}\})$:
 \begin{equation*}
\begin{split}
|\nabla^{g(t),\,k}\Ob(t)|_{g(t)}\leq e^{-\frac{C_k}{t}}.
\end{split}
\end{equation*}
\end{theo}

\begin{proof}
We recall that $U:=B_{d_0}(o,2)\setminus\overline{B_{d_0}(o,1)}$ satisfies condition $\eqref{ass-2-magic}$ of Lemma \ref{magic-comput} thanks to \cite[Lemma A.2]{DSS-II}, after reducing $T>0$ to $S>0$ once if necessary.
Condition \eqref{ass-1-magic} of Lemma \ref{magic-comput} is satisfied by assumption or, alternatively 
$|\Rm(g(t))|\leq K$  on $U:=B_{d_0}(o,2)\setminus\overline{B_{d_0}(o,1)}$ is satisfied in the case of a smooth cone away from the tip, in view of 
Remark \ref{replace2}.
Let $p\in B_{d_0}(o,2)$ such that for some $r>0$, $B_{d_0}(p,2r)\subset B_{d_0}(o,2)\setminus\overline{B_{d_0}(o,1)}$ and let $\eta$ be a Perelman type cut-off function with respect to the point $p$, $r_1:=r/2$ and $r_2:=r$ and let us compute the heat operator of $\eta |\Ob(t)|^2_{g(t)}(1+\beta t\R_{g(t)})^{-1/6},$ in the case that condition \eqref{ass-1-magic} of Lemma \ref{magic-comput} is satisfied, with the help of Lemma \eqref{magic-comput}:
\begin{equation*}
\left(\partial_t-\Delta_{g(t)}\right)\left(\eta\frac{|\Ob(t)|^2_{g(t)}}{(1+\beta t\R_{g_i(t)})^{1/6}}\right)\leq -2g(t)\left(\nabla^{g(t)}\eta,\nabla^{g(t)}\left(\frac{|\Ob(t)|^2_{g(t)}}{(1+\beta t\R_{g(t)})^{1/6}}\right)\right).
\end{equation*}
In particular, we can artificially add and subtract the scalar curvature along the Ricci flow $g(t)$ times $\eta |\Ob(t)|^2_{g(t)}(1+\beta t\R_{g(t)})^{-1/6}$ to get, thanks to the parabolic maximum principle, for $t\in[t_i,T]$, $(t_i)_{i\in \N}$ being any sequence of times going to $0$:
\begin{equation*}
\begin{split}
&\frac{|\Ob(t)|^2_{g(t)}}{(1+\beta t\R_{g(t)})^{1/6}}(p)\leq \int_MK_D(p,t,y,t_i)\eta\frac{|\Ob(t_i)|^2_{g(t_i)}}{(1+\beta t\R_{g(t_i)})^{1/6}}\,d\mu_{g(t_i)}(y)\\&-2\int_{t_i }^t\int_MK_D(p,t,y,s)g(s)\left(\nabla^{g(s)}\eta,\nabla^{g(s)}\left(\frac{|\Ob(s)|^2_{g(s)}}{(1+\beta t\R_{g(s)})^{1/6}}\right)\right)\,d\mu_{g(s)}(y)ds,
\end{split}
\end{equation*}
where $K_D(x,t,y,s)$ denotes the heat kernel associated to the heat operator with potential $\partial_t-\Delta_{g(t)}-\R_{g(t)}$ with Dirichlet boundary condition on $\partial B_{d_0}(p,r)$.

Now, \cite[Proposition $4.1$]{Lee-Tam-loc-max-ppe} ensures the following short-time Gaussian bounds for such a heat kernel: 
\begin{equation*}
K_D(p,t,y,s)\leq \frac{C}{(t-s)^\frac{n}{2}}\exp\left(-\frac{d_{g(s)}^2(p,y)}{C(t-s)}\right), \quad 0\leq s\leq t<S,
\end{equation*}
for some uniform-in-time $C\in (0,\infty)$. Indeed, $C$ can be chosen so that it only depends on an upper bound of $\sup_Mt|\Rm(g(t))|_{g(t)}$ and a lower bound on $\inf_{t\in(0,T]}t^{-1/2}\inj_{g(t)}M$. The latter quantity can be in turn estimated by a lower bound on $\AVR(g(t))$ and an upper bound of $\sup_Mt|\Rm(g(t))|_{g(t)}$.

Applying  Corollary \ref{coro-pointwise-T} for $k-n/2=1$ and using that  $|\nabla^{g(s)}\eta|_{g(s)}$ is uniformly bounded, thanks to Lemma \ref{lemma-Cutoff-Perelman},  we see 
\begin{equation*}
\begin{split}
\frac{|\Ob(t)|^2_{g(t)}}{(1+\beta t\R_{g(t)})^{1/6}}(p)&\leq C't_i\int_MK_D(p,t,y,t_i)\,d\mu_{g(t_i)}(y)\\
&\quad+C'\int_{t_i}^t\int_{B_{d_0}(p,2r)\setminus\overline{B_{d_0}(p,r/2)}}K_D(p,t,y,s)\,d\mu_{g(s)}(y)ds\\
&\leq C't_i+C'\int_{0}^t\exp\left(-\frac{C^{-1}}{(t-s)}\right)\,ds,
\end{split}
\end{equation*}
for all $t >0$  sufficiently small, 
where $C'$ is a time-independent positive constant that may vary from line to line. Here we have invoked  that $B_{g(s)}(p,r)\subset B_{d_0}(p,r+C\sqrt{s})\subset B_{d_0}(p,2r)$ according to Proposition \ref{prop:RF_basics} and if $s\leq t\leq r^2C^{-2}$ in the penultimate line.

In the case that $|\Rm(t)|\leq K$ on $U$, we use the equation 
\eqref{replaceeq} in place of  \eqref{est-T-heat-subsol} in the proof above and obtain the same results.
\end{proof}

For the sake of clarity, we give a proof of Theorem \ref{main-I} by collecting the results obtained so far:
\begin{proof}[Proof of Theorem \ref{main-I}]
Consider a solution $u$ as constructed in Proposition 
\ref{coro-dist-comp-sol} (see \eqref{eq:def-limit-solution})  solving \eqref{eq:eqn-limit-solution}. We  now invoke the results of Corollary \ref{grad-low-bd-v} and Theorem \ref{theo-exp-dec}, in combination with 
 the definition of  $v,$ and $\Ob$.    This ends the proof of Theorem \ref{main-I}.
\end{proof} 

\section{Ricci-pinched almost expanding gradient Ricci solitons}\label{sec-ric-pin}
The main result of this section is the proof of Theorem \ref{main-II}. We start by proving Theorem \ref{main-III} that is a time-dependent version of Theorem \ref{main-II} but before doing so, we state and prove first the following rigidity result for solutions as in Theorem \ref{main-III} which start from a metric cone.
\begin{theo}\label{main-iii}
Let $(M^n,g(t))_{t\in(0,T)}$ be a smooth complete Ricci flow 
such that there exist $c\in (0,\infty)$, $D_0\in (0,\infty)$ and $D_1\in (0,\infty)$ such that on $M\times(0,T)$:
\begin{equation*}
\Ric(g(t))\geq c\R_{g(t)}g(t)\geq 0,\quad |\Rm(g(t))|_{g(t)}\leq \frac{D_0}{t},\quad|\Rm(g(t))|_{g(t)}\leq D_1\R_{g(t)}.
\end{equation*}
Assume that the pointed limit in the distance sense of $(M^n,d_{g(t)},o)$ as $t$ goes to $0$ is a metric cone $(C(X),d_0,o)$.

Then $(C(X),d_0,o)$ is a smooth flat cone, and $(M^n,g(t))_{t\in(0,T)}$ is isometric to Euclidean space $(\mathbb{R}^n,\eucl)$.
\end{theo}
\begin{rk}
The proof of Theorem \ref{main-iii} only relies on a faster than polynomial decay for the tensor $\Ob$ established in Corollary \ref{coro-pointwise-T}.
\end{rk}
\begin{proof}
By interior estimates on $\Ob(t)$ from Corollary \ref{coro-pointwise-T}, we know that on $B_{d_0}(o,2)\setminus \overline{B_{d_0}(o,1)}$, $|\nabla^{g(t)}\Ob(t)|_{g(t)}\leq C_kt^k$ for all $k\geq 0$. In particular, Lemma \ref{lemma-gen-form-exp-ric-pin} ensures that:
\begin{equation*}
g(t)(\nabla^{g(t)}t\R_{g(t)}, \nabla^{g(t)}u(t))+2\Ric(g(t))(\nabla^{g(t)}u,\nabla^{g(t)}u)\leq C_kt^k,
\end{equation*}
since the gradient of $u$ is uniformly bounded by Proposition \ref{prop-grad-bd-a-priori}. Since the metric is Ricci-pinched, 
\begin{equation}\label{first-est-improv-R}
\R_{g(t)}|\nabla^{g(t)}u|^2_{g(t)}\leq C_kt^k+c|\nabla^{g(t)}(t\R_{g(t)})|_{g(t)}\leq C_kt^k+ct^{-1/2},
\end{equation}
according to Shi's interior estimates. The lower gradient bound from Corollary \ref{grad-low-bd-v} shows that once we take a limit of the functions $u^i$, one gets a uniform lower bound on the gradient of $u$ outside the apex of the cone.

In particular, \eqref{first-est-improv-R} gives an improvement on the blow-up of the curvature at the initial time: $|\Rm(g(t))|_{g(t)}\leq c(n)\R_{g(t)}\leq Ct^{-1/2}.$ Inserting this bound back to the same reasoning that led to \eqref{first-est-improv-R} gives, thanks to Shi's interior estimates (see for instance Theorem 13.1 in 
\cite{HamFor}) $|\nabla^{g(t)}(t\R_{g(t)})|_{g(t)}\leq C$, $|\Rm(g(t))|_{g(t)}\leq c(n)\R_{g(t)}\leq C.$ Iterating this reasoning a finite  number of times leads to $|\Rm(g(t))|_{g(t)}\leq C_kt^k$ for $k\in \N$ given, which is sufficient to conclude that not only 
 is the cone at time zero smooth,  but also that it is flat. Remembering that the cone has the same topology as $(M,g(t))$ for each $t\in(0,T)$ that is, it is a manifold, we conclude that the cone and $(M,g(t))$ are isometric to Euclidean space.
\end{proof}

Let us now give a proof of Theorem \ref{main-III}:
\begin{proof}[Proof of Theorem \ref{main-III}]
In order to invoke Theorem \ref{main-iii}, it suffices to blow-down the solution $(g(t))_{t>0}$ as follows. Let $(\lambda_i)_i$ be a sequence of positive numbers converging to $0$ and let $g_{i}(t):=\lambda_ig(t/\lambda_i)$ be the corresponding parabolic rescaling of the solutions of $(g(t))_{t}$ defined for $t>0$. Observe that:
\begin{equation*}
\Ric(g_i(t))\geq c\R_{g_i(t)}g_i(t)\geq 0,\quad |\Rm(g_i(t))|_{g_i(t)}\leq \frac{D_0}{t}, \quad |\Rm(g_i(t))|_{g_i(t)}\leq D_1\R_{g_i(t)}.
\end{equation*}
for uniform positive constants $c$, $D_0$ and $D_1$.

On the one hand, at $t=0$, the pointed sequence $(M,g_i(0),p)_i$ Gromov-Hausdorff converges to an asymptotic cone at infinity of $(M,g)$. Since it is assumed to be non-collapsed at all scales, i.e. $\AVR(g_i(0))=\AVR(g(0))>0$, \cite{CheegerColding_Rigidity} ensures that $\limGH_{\lambda_i\rightarrow 0}(M,g_i(0),p)=(C(X),d_0,o)$ where $d_0$ is a metric cone distance on the cone $C( X)$. On the other hand, $\AVR(g_i(t))=\AVR(g(0))>0$ for every $t>0$ and indices $i$ according to Proposition \ref{prop:RF_basics} and since $|\Rm(g_i(t))|_{g_i(t)}\leq D_0/t$ due to the scaling properties of the curvature tensor, Hamilton's compactness theorem \cite{Ham-Com} guarantees the existence of a subsequence still denoted by $(g_i(t))_i$ converging to a solution $(g_0(t))_{t>0}$ to the Ricci flow in the Cheeger-Gromov convergence. In particular, $(g_0(t))_{t>0}$ satisfies for $t>0$,
\begin{equation*}
\Ric(g_0(t))\geq c\R_{g_0(t)}g_0(t)\geq 0,\quad |\Rm(g_0(t))|_{g_0(t)}\leq \frac{D_0}{t},\quad |\Rm(g_0(t))|_{g_0(t)}\leq D_1\R_{g_0(t)}.
\end{equation*}
Moreover, Proposition \ref{prop:RF_basics} gives that $(M,d_{g_0(t)},p)$ converges to $(C(X),d_0,o)$ in the distance sense as $t$ goes to $0$. We are then in a position to apply Theorem \ref{main-iii} to conclude that $(C(X),d_0,o)$ is a smooth flat cone. In particular, $\mathcal{H}^n(B_{d_0}(o,1))=\omega_n$ where $\omega_n$ denotes the volume of the unit ball in $\mathbb{R}^n$. But since $\AVR(g(0))=\mathcal{H}^n(B_{d_0}(o,1))$ according to Cheeger-Colding's volume continuity \cite{Cheeger_notes}, one gets $\AVR(g(0))=\omega_n$ which implies by the rigidity result of Bishop-Gromov that $(M,g(0))$ is isometric to Euclidean space. Since $\AVR(g(t))=\AVR(g(0))$ for $t>0$ by Proposition \ref{prop:RF_basics} again, it also implies that the whole solution is isometric to Euclidean space. 
\end{proof}

Finally, we prove Theorem \ref{main-II}:
\begin{proof}[Proof of Theorem \ref{main-II}]
Let $(M^n,g)$ be a complete Riemannian manifold such that it is PIC1 pinched. Then by the results of \cite{Lee-Top-PIC2}, there exists a complete solution to Ricci flow $(M^n,g(t))_{t>0}$ such that it is uniformly PIC1 pinched with curvature operator $\Rm(g(t))$ bounded by $t^{-1}$. In particular, it is uniformly Ricci-pinched. We summarize the properties of this flow as follows:
\begin{equation*}
\Ric(g(t))\geq c\R_{g(t)}g(t)\geq 0,\quad |\Rm(g(t))|_{g(t)}\leq \frac{D_0}{t}, \quad |\Rm(g(t))|_{g_0(t)}\leq D_1\R_{g(t)}.
\end{equation*}
for some time-independent positive constants $c$, $D_0$ and $D_1$. The third property follows from Section \ref{sec-curv-constraints}. Moreover, as we assume $\AVR(g)=\AVR(g(0))>0$ then $\AVR(g(t))=\AVR(g)>0$ for $t>0$ by Proposition \ref{prop:RF_basics}. Then Theorem \ref{main-III} allows us to conclude that $(M,g(t))_{t>0}$ and hence $(M,g)$ is isometric to Euclidean space.
\end{proof}
\newpage
\begin{appendix}
\section{Perelman cut-off functions}
The following lemma is essentially taken from \cite[Lemma $7.1$]{Sim-Top-3d}:
\begin{lemma}\label{lemma-Cutoff-Perelman}
Let $c_0>0$, $0<r_1<r_2$ and $n\in\mathbb{N}$ arbitrary. Suppose that $(M^n,g(t))$ is a smooth Ricci flow for $t\in[0,T)$, and $x_0\in M$ satisfies $B_{g(t)}(x_0,r_2)\subset \subset M$ for all $t\in[0,T)$. We assume further that
\begin{equation}\label{cond-Cutoff-Perelman}
\Ric(g(t))\leq \frac{c_0(n-1)}{t}g(t), \quad\text{on $B_{g(t)}(x_0,\sqrt{t})$,} 
\end{equation}
for all $t\in(0,T)$. Then there exist positive constants $\hat{T}=\hat{T}(c_0,r_1,r_2,n)>0$, $k=k(r_1,r_2)>0$, $V=V(r_1,r_2)>0$ and a locally Lipschitz continuous function $\eta:M\times[0,\hat{T})\rightarrow \mathbb{R}$ such that
\begin{enumerate}
\item $\eta(\cdot,t)=e^{-kt}$ on $B_{g(t)}(x_0,r_1)$ and $\eta(\cdot,t)=0$ outside $B_{g(t)}(x_0,r_2)$ for all $t\in[0,\hat{T})\cap [0,T)$ and $\eta(x,t)\in[0,e^{-kt}]$ for all $(x,t)\in M\times [0,\hat{T})\cap [0,T)$,\\
\item $|\nabla^{g(t)}\eta|^2_{g(t)}\leq V\eta$, $\partial_t\eta\geq -V/\sqrt{t}$ a.e. and $\eta$ is a subsolution to the heat equation in the following barrier sense. For any $(x,t)\in M\times (0,\hat{T})\cap (0,T)$, we can find a neighborhood $\mathcal{O}$ of $(x,t)$ and a smooth function $\hat{\eta}:\mathcal{O}\rightarrow\mathbb{R}_+$ such that $\hat{\eta}\leq\eta$ on $\mathcal{O}$, $\hat{\eta}(x,t)=\eta(x,t)$ and $\partial_t\hat{\eta}(x,t)\leq \Delta_{g(t)}\hat{\eta}(x,t)$ and $|\nabla^{g(t)}\hat{\eta}|^2_{g(t)}(x,t)\leq V\hat{\eta}(x,t)$.\\

\item The function $\eta$ satisfies $\partial_t\eta\leq \Delta_{g(t)}\eta$ in the weak sense: see \cite[Section $9.4$]{Chow-Lu-Ni-I} for instance for a proof.
\end{enumerate}
\end{lemma}

\begin{rk}
We will use Lemma \ref{lemma-Cutoff-Perelman} in the case where $r_1$ and $r_2$ are universally proportional to a given scale, say $r_1= \alpha  r, r_2 =r>0$. As pointed out in \cite[Lemma $7.1$]{Sim-Top-3d},  $\hat{T}$ can be chosen  to be   $ \hat{T} = B(n)  (1-\al)^2 \geq \frac{r^2}{c_0^2},$  and $k=   \frac{B(n)}{(1-\al)^2 r^2}.$
\end{rk}
\section{Curvature constraints}\label{sec-curv-constraints}
This short section is devoted to the proof of a result that is essentially contained in \cite{Mic-Wan}: see also \cite[Lemma A.1]{Lee-Top-PIC2} for PIC1 metrics.

Recall that a Riemannian metric $g$ is said to have non-negative isotropic curvature if for any orthonormal four-frames $(e_i)_{1\leq i\leq 4}$,
\begin{equation*}
\Rm(g)_{1331}+\Rm(g)_{1441}+\Rm(g)_{2332}+\Rm(g)_{2442}+2\Rm(g)_{1234}\geq 0,
\end{equation*}
where $\Rm(g)_{ijkl}:=g\left(\Rm(g)(e_i,e_j)e_k,e_l\right).$
In particular, by switching the roles of $e_3$ and $e_4$, one gets the following condition:
\begin{equation*}
\Rm(g)_{1331}+\Rm(g)_{1441}+\Rm(g)_{2332}+\Rm(g)_{2442}\geq 0.
\end{equation*}
By renaming indices, one ends up with the following consequence:
\begin{equation*}
\Rm(g)_{ikki}+\Rm(g)_{illi}+\Rm(g)_{jkkj}+\Rm(g)_{jllj}\geq 0.
\end{equation*}
By tracing over the index $l$ different from $i$, $j$, $k$:
\begin{equation*}
\begin{split}
(n-4)(\Rm(g)(e_i,e_k,e_k,e_i)&+\Rm(g)(e_j,e_k,e_k,e_j))\\
&-2\Rm(g)(e_i,e_j,e_j,e_i)+\Ric(g)(e_i,e_i)+\Ric(g)(e_j,e_j)\geq 0.
\end{split}
\end{equation*}
Tracing once more over the index $k$ different from $i$, $j$:
\begin{equation*}
(2n-6)\left(\Ric(g)(e_i,e_i)+\Ric(g)(e_j,e_j)-2\Rm(g)(e_i,e_j,e_j,e_i)\right)\geq 0.
\end{equation*}
In particular, if $\Ric(g)\geq 0$ then $\Ric(g)\leq \R_gg$ which implies that for all orthonormal two-frames $\{e_1,e_2\}$:
\begin{equation}
\R_g\geq \Rm(g)(e_i,e_j,e_j,e_i).\label{upp-bd-sec-NIC}
\end{equation}
Finally, for two distinct indices $i$ and $j$:
\begin{equation*}
\begin{split}
\R_g&=\sum_{k=1}^n\Ric(g)(e_k,e_k)=\sum_{k\neq i,j}\Ric(g)(e_k,e_k)+\Ric(g)(e_i,e_i)+\Ric(g)(e_j,e_j)\\
&\leq (n-2)\R_g+\Ric(g)(e_i,e_i)+\Ric(g)(e_j,e_j)\\
&\leq (n-2)\R_g+2\Rm(g)(e_i,e_j,e_j,e_i)+\sum_{k\neq i,j}\Rm(g)(e_i,e_k,e_k,e_i)+\sum_{k\neq i,j}\Rm(g)(e_j,e_k,e_k,e_j)\\
&\leq 3(n-2)\R_g+2\Rm(g)(e_i,e_j,e_j,e_i),
\end{split}
\end{equation*}
where we have used the upper bound \eqref{upp-bd-sec-NIC}.
Let us summarize this discussion in the following proposition essentially due to Micaleff and Wang \cite{Mic-Wan}:
\begin{prop}\label{prop-bds-curv}
Let $(M^n,g)$, $n\geq 4$, be a Riemannian manifold with non-negative isotropic curvature and non-negative Ricci curvature. Then for all orthonormal two-frames $\{e_1,e_2\}$:
\begin{equation}
-\frac{(3n-7)}{2}\R_g\leq\Rm(g)(e_i,e_j,e_j,e_i)-\R_g\leq 0.\label{est-upp-inf-bd-sec}
\end{equation}
In particular, if $\R_g=0$ then the metric $g$ is flat.
\end{prop}
As a corollary, Proposition \ref{prop-bds-curv} holds true if $(M^n,g)$ has weak PIC1 which is implied if $(M^n,g)$ has $2$-non-negative curvature operator.

\section{Sharp Poisson regularization of $d_0^2$} \label{appendix-Poisson}

If $\varepsilon>0$, recall that a metric ball $B_d(x,r)\subset X$ is $(k,\varepsilon^2)$-symmetric if there exists a $k$-symmetric metric cone $X':=\mathbb{R}^k\times C(Z)$ with $x'$ a vertex of $X'$ such that $$d_{GH}(B_{d}(x,r),B_{d'}(x',r))<\varepsilon r,$$
where $d'$ is the product metric on $X'$. The following statement employs the local entropy $\mathcal{W}_{\varepsilon}^{\delta}$ as introduced in 
\cite[Definition 4.19]{Che-Jia-Nab}, compare the earlier work of Cheeger-Colding (see \cite{Cheeger_notes}). 

\begin{theo}\cite[Theorem 6.3]{Che-Jia-Nab} \label{che-jia-nab-maps}
For  $(M^n,g,p),$ and $\varepsilon,v >0$ , there exists a $\de_1=\de_1(n,v,\varepsilon)$ such that the following holds for all $\de \leq \de_1$.  
Assume  $\Ric(g)\geq -(n-1)\delta^2$ with $\vol_gB_g(p,\delta^{-1})>v\delta^{-n}>0.$ If    $B_g(x,r)\Subset B_g(p,5)$   and $B_g(x,r\delta^{-1})$ is $(0,\delta^2)$-symmetric, then there exists a function $u:B_g(x,2r)\rightarrow \mathbb{R}$ such that:\\[-1ex]
\begin{enumerate}
\item (Poisson equation) $\Delta_gu=\frac{n}{2}.$\\[-1ex]

\item (Hessian bounds on $u$)
\begin{equation*}
\begin{split}
\dashint_{B_g(x,2r)}\bigg(\left|\nabla^{g,2}u-\frac{g}{2}\right|^2_g+\Ric(g)(\nabla^gu,\nabla^gu)&+2(n-1)\delta^2|\nabla^gu|^2_g\bigg)\,d\mu_g\\
&\leq C(n,v)\left|\mathcal{W}^{\delta}_{r^2}(x)-\mathcal{W}^{\delta}_{2r^2}(x)\right|.
\end{split}
\end{equation*}
\item (Sharp $L^2$ gradient bound) 
\begin{equation*}
\dashint_{B_g(x,2r)}\left||\nabla^gu|^2_g-u\right|^2\,d\mu_g\leq C(n,v)r^4\left|\mathcal{W}^{\delta}_{r^2}(x)-\mathcal{W}^{\delta}_{2r^2}(x)\right|.
\end{equation*}
\item ($L^{\infty}$ upper gradient bound) $|\nabla^gu|_g\leq C(n,v)r.$\\[-1ex]

\item ($L^{\infty}$ closeness)
\begin{equation*}
 \sup_{B_g(x,2r)}\left|u-\frac{d_g(x,\cdot)^2}{4}\right|\leq \varepsilon r^2.
\end{equation*}
\end{enumerate}
\end{theo}

Assume $(M^n,g)$ satisfies $\Ric(g)\geq -(n-1)\kappa$ for some $\kappa \geq 0$. Recall that the volume ratio $\mathcal{V}^\kappa_r(x)$ at a point $x\in M$ is defined as:
\begin{equation*}
\mathcal{V}^\kappa_g(x,r):=\frac{\vol_gB_g(x,r)}{\vol_{-\kappa}\mathbb{B}(r)},
\end{equation*}
where $\vol_{-\kappa}\mathbb{B}(r)$ denotes the volume of an $r$-ball $\mathbb{B}(r)$ in a simply connected space of constant curvature $-\kappa$.
 We further recall the following behavior of the local entropy.
\begin{theo}\cite[Theorem 4.21]{Che-Jia-Nab} 
\label{theo-che-jia-nab-vol-entropy}
For  $(M^n,g,p)$, and positive constants $\varepsilon$ and $v$, there exists a $\de_2=\de_2(n,v,\varepsilon)>0$ such that the following holds for all $\de \leq \de_2$. 
Assume  $\Ric(g)\geq -(n-1)\delta^2$ with $\vol_gB_g(p,r)>vr^n>0,$ for all $0<r\leq \delta^{-1}$. 

Then for all $x\in B_g(p,\delta^{-1}/2)$ and $s\leq \delta^{-2}$, the local $\mathcal{W}^{\delta}_{s}$-entropy satisfies the following:
If $s\in (0,10)$ and  $|\mathcal{V}^0_{g}(x,\sqrt{s}\delta^{-1})-\mathcal{V}^0_{g}(x,\sqrt{s}\delta)|\leq\delta$  
    then $\left|\mathcal{W}^{\delta}_{s}(x)-\log \mathcal{V}^{\delta^2}_{g}(x,\sqrt{s})\right|\leq\varepsilon.$
\end{theo}
\end{appendix}
%\newpage 
\bigskip\bigskip
\newcommand{\etalchar}[1]{$^{#1}$}
\providecommand{\bysame}{\leavevmode\hbox to3em{\hrulefill}\thinspace}
\providecommand{\MR}{\relax\ifhmode\unskip\space\fi MR }
% \MRhref is called by the amsart/book/proc definition of \MR.
\providecommand{\MRhref}[2]{%
  \href{http://www.ams.org/mathscinet-getitem?mr=#1}{#2}
}
\providecommand{\href}[2]{#2}


\begin{thebibliography}{BCRW19}

\bibitem[BC23]{Bamler-Chen}
Richard~H. {Bamler} and Eric {Chen}, \emph{{Degree theory for 4-dimensional
  asymptotically conical gradient expanding solitons}}, arXiv e-prints (2023),
  arXiv:2305.03154.

\bibitem[BCRW19]{Bam-Cab-Riv-Wil}
Richard~H. Bamler, Esther Cabezas-Rivas, and Burkhard Wilking, \emph{The
  {R}icci flow under almost non-negative curvature conditions}, Invent. Math.
  \textbf{217} (2019), no.~1, 95--126.

\bibitem[Bre09]{Bre-Har-RF}
Simon Brendle, \emph{A generalization of {H}amilton's differential {H}arnack
  inequality for the {R}icci flow}, J. Differential Geom. \textbf{82} (2009),
  no.~1, 207--227. \MR{2504774}

\bibitem[BS09]{Brendle-Schoen09}
Simon Brendle and Richard Schoen, \emph{Sphere theorems in geometry}, Surveys
  in differential geometry. {V}ol. {XIII}. {G}eometry, analysis, and algebraic
  geometry: forty years of the {J}ournal of {D}ifferential {G}eometry, Surv.
  Differ. Geom., vol.~13, Int. Press, Somerville, MA, 2009, pp.~49--84.
  \MR{2537082}

\bibitem[CC96]{CheegerColding_Rigidity}
Jeff Cheeger and Tobias~H. Colding, \emph{Lower bounds on {R}icci curvature and
  the almost rigidity of warped products}, Ann. of Math. (2) \textbf{144}
  (1996), no.~1, 189--237.

\bibitem[CCG{\etalchar{+}}07]{Chow-Lu-Ni-I}
Bennett Chow, Sun-Chin Chu, David Glickenstein, Christine Guenther, James
  Isenberg, Tom Ivey, Dan Knopf, Peng Lu, Feng Luo, and Lei Ni, \emph{The
  {R}icci flow: techniques and applications. {P}art {I}}, Mathematical Surveys
  and Monographs, vol. 135, American Mathematical Society, Providence, RI,
  2007, Geometric aspects. \MR{2302600}

\bibitem[CCG{\etalchar{+}}10]{Chow2010TheRF}
\bysame, \emph{The {R}icci flow: techniques and applications. {P}art {III}.
  {G}eometric-analytic aspects}, Mathematical Surveys and Monographs, vol. 163,
  American Mathematical Society, Providence, RI, 2010. \MR{2604955}

\bibitem[CDS19]{CDS}
R.~J. Conlon, A.~Deruelle, and S.~Sun, \emph{Classification results for
  expanding and shrinking gradient {K}\"ahler-{R}icci solitons}, to appear in
  Geom.~Topol., arXiv:1904.00147 (2019).

\bibitem[Che01]{Cheeger_notes}
Jeff Cheeger, \emph{Degeneration of {R}iemannian metrics under {R}icci
  curvature bounds}, Lezioni Fermiane (Fermi Lectures), Scuola Normale
  Superiore, Pisa, 2001.

\bibitem[CJN21]{Che-Jia-Nab}
Jeff Cheeger, Wenshuai Jiang, and Aaron Naber, \emph{Rectifiability of singular
  sets of noncollapsed limit spaces with {R}icci curvature bounded below}, Ann.
  of Math. (2) \textbf{193} (2021), no.~2, 407--538. \MR{4226910}

\bibitem[CLN06a]{Chow-Lu-Ni}
Bennett Chow, Peng Lu, and Lei Ni, \emph{Hamilton's {R}icci flow}, Graduate
  Studies in Mathematics, vol.~77, American Mathematical Society, Providence,
  RI; Science Press Beijing, New York, 2006. \MR{2274812}

\bibitem[CLN06b]{CLN}
\bysame, \emph{Hamilton's {R}icci flow}, Graduate Studies in Mathematics,
  vol.~77, American Mathematical Society, Providence, RI; Science Press
  Beijing, New York, 2006.

\bibitem[CZ00]{Chen-Zhu}
Bing-Long Chen and Xi-Ping Zhu, \emph{Complete {R}iemannian manifolds with
  pointwise pinched curvature}, Invent. Math. \textbf{140} (2000), no.~2,
  423--452.

\bibitem[Der16]{Der-Smo-Pos-Cur-Con}
Alix Deruelle, \emph{Smoothing out positively curved metric cones by {R}icci
  expanders}, Geom. Funct. Anal. \textbf{26} (2016), no.~1, 188--249.

\bibitem[DSS22a]{DSS-II}
Alix {Deruelle}, Felix {Schulze}, and Miles {Simon}, \emph{{Initial stability
  estimates for Ricci flow and three dimensional Ricci-pinched manifolds}},
  arXiv e-prints (2022), arXiv:2203.15313.

\bibitem[DSS22b]{DSS-1}
Alix Deruelle, Felix Schulze, and Miles Simon, \emph{On the regularity of
  {R}icci flows coming out of metric spaces}, J. Eur. Math. Soc. (JEMS)
  \textbf{24} (2022), no.~7, 2233--2277. \MR{4413766}

\bibitem[{Gri}92]{Grig92}
A.~A. {Grigor'yan}, \emph{{The Heat Equation on Noncompact Riemannian
  Manifolds}}, Sbornik: Mathematics \textbf{72} (1992), no.~1, 47--77.

\bibitem[Ham95a]{HamFor}
Richard Hamilton, \emph{The formation of singularities in the {R}icci flow},
  Surveys in differential geometry, {V}ol. {II} ({C}ambridge, {MA}, 1993), Int.
  Press, Cambridge, MA, 1995, pp.~7--136.

\bibitem[Ham95b]{Ham-Com}
Richard~S. Hamilton, \emph{A compactness property for solutions of the {R}icci
  flow}, Amer. J. Math. \textbf{117} (1995), no.~3, 545--572. \MR{1333936}

\bibitem[HK23]{Hui-Koe}
Gerhard {Huisken} and Thomas {Koerber}, \emph{{Inverse mean curvature flow and
  Ricci-pinched three-manifolds}}, arXiv e-prints (2023), arXiv:2305.04702.

\bibitem[{Lot}19]{Lott-Ricci-pinched}
John {Lott}, \emph{{On 3-manifolds with pointwise pinched nonnegative Ricci
  curvature}}, To be published in Mathematische Annalen (2019),
  arXiv:1908.04715.

\bibitem[LT22a]{Lee-Tam-loc-max-ppe}
Man-Chun Lee and Luen-Fai Tam, \emph{Some local maximum principles along
  {R}icci flows}, Canad. J. Math. \textbf{74} (2022), no.~2, 329--348.
  \MR{4410993}

\bibitem[LT22b]{Lee-Top-PIC2}
Man-Chun {Lee} and Peter~M. {Topping}, \emph{{Manifolds with PIC1 pinched
  curvature}}, arXiv e-prints (2022), arXiv:2211.07623.

\bibitem[LT22c]{Lee-Top-3d}
\bysame, \emph{{Three-manifolds with non-negatively pinched Ricci curvature}},
  arXiv e-prints (2022), arXiv:2204.00504.

\bibitem[M{\'a}x14]{Maximo}
Davi M{\'a}ximo, \emph{On the blow-up of four-dimensional {R}icci flow
  singularities}, J. Reine Angew. Math. \textbf{692} (2014), 153--171.
  \MR{3274550}

\bibitem[MW93]{Mic-Wan}
Mario~J. Micallef and McKenzie~Y. Wang, \emph{Metrics with nonnegative
  isotropic curvature}, Duke Math. J. \textbf{72} (1993), no.~3, 649--672.
  \MR{1253619}

\bibitem[NW07]{Ni-Wu}
Lei Ni and Baoqiang Wu, \emph{Complete manifolds with nonnegative curvature
  operator}, Proc. Amer. Math. Soc. \textbf{135} (2007), no.~9, 3021--3028.
  \MR{2511306}

\bibitem[SC92]{Saloff-Coste-Riem}
Laurent Saloff-Coste, \emph{Uniformly elliptic operators on {R}iemannian
  manifolds}, J. Differential Geom. \textbf{36} (1992), no.~2, 417--450.
  \MR{1180389}

\bibitem[SS13]{Sch-Sim}
Felix Schulze and Miles Simon, \emph{Expanding solitons with non-negative
  curvature operator coming out of cones}, Math. Z. \textbf{275} (2013),
  no.~1-2, 625--639.

\bibitem[ST21]{SiTo2}
Miles Simon and Peter~M. Topping, \emph{Local mollification of {R}iemannian
  metrics using {R}icci flow, and {R}icci limit spaces}, Geom. Topol.
  \textbf{25} (2021), no.~2, 913--948.

\bibitem[ST22]{Sim-Top-3d}
\bysame, \emph{Local control on the geometry in 3{D} {R}icci flow}, J.
  Differential Geom. \textbf{122} (2022), no.~3, 467--518. \MR{4544560}

\bibitem[{Top}23]{Top-Sur}
Peter~M. {Topping}, \emph{{Ricci flow and PIC1}}, arXiv e-prints (2023),
  arXiv:2309.00596.

\bibitem[Yok08]{Yokota}
Takumi Yokota, \emph{Curvature integrals under the {R}icci flow on surfaces},
  Geom. Dedicata \textbf{133} (2008), 169--179.

\end{thebibliography}
\end{document}